\pdfoutput=1
\documentclass[a4paper,11pt]{amsart}

\usepackage[margin=1in,includehead,includefoot]{geometry}
\usepackage[utf8]{inputenc}
\usepackage[T1]{fontenc}
\usepackage{lmodern,microtype}
\usepackage{mathtools,amssymb,amsthm}
\usepackage{enumitem}
\usepackage{tikz,pgfplots}
\usetikzlibrary{decorations.markings}
\usetikzlibrary{fillbetween}
\usetikzlibrary{external}
\usepackage[unicode]{hyperref}
\usepackage{bookmark}

\hypersetup{colorlinks=true,linkcolor=black,citecolor=black,filecolor=black,urlcolor=black}

\makeatletter
\def\@makefntext{\leavevmode\llap{\@makefnmark\kern 3pt}}
\let\citationorig\citation
\def\citation#1{\citationorig{#1}\@for\@tempa:=#1\do{\@ifundefined{cit@\@tempa}{\global\@namedef{cit@\@tempa}{}}{}}}
\let\bibitemorig\bibitem
\def\bibitem#1{\@ifundefined{cit@#1}{\typeout{LaTeX Warning: Unused bibitem `#1'}}{}\bibitemorig{#1}}
\let\old@setaddresses\@setaddresses
\def\@setaddresses{\bigskip{\parindent 0pt\let\scshape\relax\let\ttfamily\relax\old@setaddresses}}
\makeatother

\def\periodsf{\spacefactor 3000 \space}

\newtheorem*{mainthm}{Theorem}
\newtheorem{lemma}{Lemma}

\renewenvironment{enumerate}{\begin{enumorig}[label=\textup{(\arabic*)}, noitemsep, topsep=1.5mm plus 1.5mm, leftmargin=*]}{\end{enumorig}}

\renewenvironment{itemize}{\begin{itemorig}[label=\textbullet, noitemsep, topsep=1.5mm plus 1.5mm, labelsep=.6em, labelindent=.2em, leftmargin=*]}{\end{itemorig}}

\let\int\undefined
\DeclareMathOperator{\int}{int}
\DeclareMathOperator{\ext}{ext}

\def\famF{\mathcal{F}}
\def\famG{\mathcal{G}}
\def\famH{\mathcal{H}}
\def\famS{\mathcal{S}}
\def\famU{\mathcal{U}}

\def\setN{\mathbb{N}}
\def\setR{\mathbb{R}}

\let\leq\leqslant
\let\geq\geqslant
\let\setminus\smallsetminus
\let\Omega\varOmega
\let\Theta\varTheta

\allowdisplaybreaks

\hypersetup{
  pdftitle={Outerstring graphs are χ-bounded},
  pdfauthor={Alexandre Rok, Bartosz Walczak},
}

\title[Outerstring graphs are $\chi$-bounded]{\boldmath Outerstring graphs are $\chi$-bounded}
\author{Alexandre Rok\and Bartosz Walczak}

\address[Alexandre Rok]{Department of Mathematics, Ben-Gurion University of the Negev, Be'er Sheva 84105, Israel}
\email{\href{mailto:roky3090@gmail.com}{roky3090@gmail.com}}

\address[Bartosz Walczak]{Department of Theoretical Computer Science, Faculty of Mathematics and Computer Science, Jagiellonian University, Kraków, Poland}
\email{\href{mailto:walczak@tcs.uj.edu.pl}{walczak@tcs.uj.edu.pl}}

\thanks{A preliminary version of this paper appeared in: \href{https://doi.org/10.1145/2582112.2582115}{Siu-Wing Cheng and Olivier Devillers (eds.), \emph{30th Annual Symposium on Computational Geometry (SoCG 2014)}, pp.~136--143, ACM, New York, 2014}.}
\thanks{Alexandre Rok was partially supported by Swiss National Science Foundation grants 200020-144531 and 200021-137574 and by Israel Science Foundation grant 1136/12.
Bartosz Walczak was partially supported by Swiss National Science Foundation grant 200020-144531, by Ministry of Science and Higher Education of Poland grant 884/N-ESF-EuroGIGA/10/2011/0 within ESF EuroGIGA project GraDR, and by National Science Center of Poland grant 2015/17/D/ST1/00585.}

\begin{document}

\begin{abstract}
An outerstring graph is an intersection graph of curves that lie in a common half-plane and have one endpoint on the boundary of that half-plane.
We prove that the class of outerstring graphs is $\chi$-bounded, which means that their chromatic number is bounded by a function of their clique number.
This generalizes a series of previous results on $\chi$-boundedness of outerstring graphs with various additional restrictions on the shape of curves or the number of times the pairs of curves can cross.
The assumption that each curve has an endpoint on the boundary of the half-plane is justified by the known fact that triangle-free intersection graphs of straight-line segments can have arbitrarily large chromatic number.
\end{abstract}

\maketitle

\section{Introduction}

\subsection*{Overview}

The \emph{intersection graph} of a family of sets $\famF$ is the graph with vertex set $\famF$ and edge set comprising the pairs of members of $\famF$ that intersect.
A \emph{curve} is a homeomorphic image of the real interval $[0,1]$ in the plane.
We consider finite families $\famF$ of curves in a closed half-plane such that each curve $c\in\famF$ has exactly one point on the boundary of the half-plane and that point is an endpoint of $c$.
Such families of curves are called \emph{grounded}, and their intersection graphs are known as \emph{outerstring graphs}.

For a graph $G$, we let $\chi(G)$ denote the chromatic number of $G$ and $\omega(G)$ denote the clique number of $G$ (the maximum size of a clique in $G$).
A class of graphs $\famG$ is \emph{$\chi$-bounded} if there is a function $f\colon\setN\to\setN$ such that every graph $G\in\famG$ satisfies $\chi(G)\leq f(\omega(G))$.

The main result of this paper is that the class of outerstring graphs is $\chi$-bounded.
Specifically, we establish the following bound.

\begin{mainthm}
Every outerstring graph\/ $G$ satisfies\/ $\chi(G)\leq f(\omega(G))$, where\/ $f(\omega)=\smash[t]{2^{O(2^{\omega(\omega-1)/2})}}$.
\end{mainthm}

Outerstring graphs are special instances of \emph{string graphs}---intersection graphs of generic curves in the plane.
It is known, however, that the class of string graphs is not $\chi$-bounded \cite{PKK+14}.

\subsection*{String graphs and outerstring graphs}

Mathematical study of string graphs evolved from the works of Benzer \cite{Ben59} on the topology of genetic structures and of Sinden \cite{Sin66} on electrical networks realizable by printed circuits.
In particular, Sinden \cite{Sin66} showed that not all graphs are strings graphs while all planar graphs are.
Ehrlich, Even, and Tarjan \cite{EET76} (seemingly unaware of these earlier works) defined string graphs formally and proved that deciding whether a string graph is properly $k$-colorable is NP-complete for every $k\geq 3$.
The first thorough treatment of string graphs is by Kratochvíl, Goljan, and Kučera \cite{KGK-book}.
In particular, they introduced the concepts of outerstring graphs and \emph{double outerstring graphs}---intersection graphs of curves that lie between and connect two fixed parallel lines.
Variants of the latter concept appeared earlier in \cite{GRU83,Lov83}; all of them were shown to define precisely the class of incomparability graphs of partially ordered sets (see also \cite{FP12b}).
The term ``(double) outerstring graph'' was first used in \cite{Kra91a}.
Therein, outerstring graphs are defined as intersection graphs of curves that lie in a closed disc and have one endpoint (and no other points) on the boundary of that disc, but that definition is equivalent (up to isomorphism) to the one we use.%
\footnote{Outerstring realizations corresponding to the two definitions can be mapped to each other by a homeomorphism transforming a closed half-plane to a closed disc without one boundary point.}
Relations between various subclasses of the class of outerstring graphs were investigated in \cite{CJ17,CFM+17,JT-arxiv}.

The complexity and even decidability of string graph recognition was a long-standing open problem.
It is easy to prove that every finite string graph has a realization as the intersection graph of curves in the plane with finitely many intersection points (see \cite{KGK-book}), but it is far from clear whether the number of intersection points can be bounded by any function of the number of vertices.
Kratochvíl and Matoušek \cite{KM91} constructed string graphs on $n$ vertices that require $\smash[t]{2^{\Omega(n)}}$ intersection points in any realization, and they conjectured that there is always a realization with at most $\smash[t]{2^{\operatorname{poly}(n)}}$ intersection points.
Kratochvíl \cite{Kra91b} proved that string graph recognition is NP-hard.
Much later, Pach and Tóth \cite{PT02} and independently Schaefer and Štefankovič \cite{SS04} confirmed the above-mentioned conjecture from \cite{KM91}, thus proving that string graph recognition is decidable.
Finally, Schaefer, Sedgwick, and Štefankovič \cite{SSS03} proved that the problem belongs to NP\@.

An alternative proof of NP-hardness of string graph recognition was provided by Middendorf and Pfeiffer \cite{MP93}.
It proceeds in two steps.
The first one is a polynomial-time reduction of the NP-complete problem of recognizing Hasse diagrams \cite{Bri93,NR95} to recognition of cylinder graphs (intersection graphs of curves connecting two fixed concentric circles), which form a subclass of the outerstring graphs.%
\footnote{Specifically, it is proved in \cite[Corollary 2.7]{MP93} that a graph is a Hasse diagram if and only if it is triangle-free and its complement is a cylinder graph.}
The second step is a polynomial-time reduction of the latter problem to string graph recognition.
Kratochvíl \cite{Kra91a} observed that the reduction used for the first step also proves NP-hardness of outerstring graph recognition.%
\footnote{This is mentioned in \cite{Kra91a} in \emph{Note added in proof} and follows from the above and the fact that triangle-free complements of outerstring graphs are Hasse diagrams, which is a direct consequence of Theorem~1 in \cite{Sin66}.}
A straightforward modification of the second reduction from \cite{MP93} gives rise to a polynomial-time reduction of outerstring graph recognition to string graph recognition, thus showing (in view of the above-mentioned result of \cite{SSS03}) that outerstring graph recognition belongs to NP\@.%
\footnote{Given a graph $G$, add a vertex $r$ non-adjacent to $G$ and, for every vertex $v$ of $G$, a vertex $r_v$ adjacent only to $r$ and $v$.
The graph thus obtained is a string graph if and only if $G$ is an outerstring graph; see \cite[Lemma~1]{BBD18}.}
Biedl, Biniaz, and Derka \cite{BBD18} constructed $n$-vertex outerstring graphs requiring $\smash[t]{2^{\Omega(n)}}$ intersection points in any realization.
Double outerstring\linebreak[4] graphs (incomparability graphs) can be recognized in polynomial time \cite{Gol77}.

Pach and Tóth \cite{PT06} proved that the number of string graphs on $n$ vertices is $\smash[t]{2^{\left(\frac{3}{4}+o(1)\right)\binom{n}{2}}}$.
Although this is not stated in \cite{PT06}, an entirely analogous argument shows that the number of outerstring graphs on $n$ vertices is $\smash[t]{2^{\left(\frac{2}{3}+o(1)\right)\binom{n}{2}}}$.%
\footnote{One needs the fact that the graph $G_4$ defined in \cite{PT06} is not an outerstring graph.
This is because any outerstring representation of $G_4$ could be extended to a string representation of $G_5$, contradicting Lemma 3.2 in \cite{PT06}.}
The number of double outerstring graphs (incomparability graphs) on $n$ vertices is $\smash[t]{2^{\left(\frac{1}{2}+o(1)\right)\binom{n}{2}}}$, which follows from the result of Kleitman and Rothschild \cite{KR70}.
Janson and Uzzell \cite{JU17} extended these results, determining the limiting behavior of sequences of random string graphs, outerstring graphs, and incomparability graphs.

Every intersection graph of compact arc-connected sets in the plane is a string graph, because every such set can be approximated by its ``filling curve'' to arbitrary precision \cite{Sin66}.

\subsection*{Coloring geometric intersection graphs}

The study of chromatic number of intersection graphs of geometric objects in the plane was initiated by Asplund and Grünbaum \cite{AG60}, who proved that the class of rectangle graphs (intersection graphs of axis-parallel rectangles in the plane) is $\chi$-bounded.
Their proof gives the bound $\chi\leq 4\omega^2-4\omega$, later improved to $\chi\leq 3\omega^2-2\omega-1$ by Hendler \cite{Hen98}.
Chalermsook \cite{Cha11} proved the bound $\chi=O(\omega\log\omega)$ for the special case that none of the rectangles is contained in another.
Kostochka \cite{Kos04} claimed a construction of rectangle graphs achieving $\chi=3\omega$, and this is the best lower bound known to date.

Gyárfás \cite{Gya85} proved that circle graphs (intersection graphs of chords of a circle) satisfy $\chi\leq 2^\omega\omega^2(2^\omega-2)$.
This bound was later improved to $\chi\leq 2^\omega\omega(\omega+2)$ by Kostochka \cite{Kos88}, to $\chi\leq 50\cdot 2^\omega-32\omega-64$ by Kostochka and Kratochvíl \cite{KK97} (for the more general class of intersection graphs of polygons inscribed in a circle), and finally to $\chi\leq 21\cdot 2^\omega-24\omega-24$ by Černý \cite{Cer07}.
Using the probabilistic method, Kostochka \cite{Kos88,Kos04} proved that there exist circle graphs satisfying $\chi\geq\frac{1}{2}\omega\ln\omega-\omega$, and this is the best lower bound known to date.

Motivated by practical applications to channel assignment, Peeters \cite{Pee91} showed that intersection graphs of unit discs satisfy $\chi\leq 3\omega-2$, whereas Malesińska, Piskorz, and Weißenfels \cite{MPW98} proved $\chi\leq 6\omega-6$ for intersection graphs of discs of arbitrary sizes.
More generally, Kim, Kostochka, and Nakprasit \cite{KKN04} showed that intersection graphs of homothets of a fixed convex compact set in the plane satisfy $\chi\leq 6\omega-6$ while intersection graphs of translates of such a set satisfy $\chi\leq 3\omega-2$.
These results actually show a property stronger than $\chi$-boundedness, namely, that average degree is bounded by a function of $\omega$.
In fact, string graphs excluding the complete bipartite graph $K_{t,t}$ as a subgraph have average degree $O(t\log t)$; this was shown by Fox and Pach \cite{FP10,FP14} using a separator theorem conjectured in \cite{FP10,FP14} and later proved by Lee \cite{Lee17}.

Colorings of outerstring graphs in various restricted settings were first considered by McGuinness \cite{McG96,McG00}.
In \cite{McG96}, he proved that intersection graphs of L-shapes\footnote{An \emph{L-shape} consists of a horizontal and a vertical segment joined to form the letter~L.} intersecting a common horizontal line satisfy $\chi=\smash[t]{2^{O(4^\omega)}}$.
In \cite{McG00}, he proved that triangle-free intersection graphs of simple\footnote{A family of compact arc-connected sets is \emph{simple} if the intersection of any subset of them is arc-connected.} grounded\footnote{A family of compact arc-connected sets is \emph{grounded} if the sets are contained in a half-plane and the intersection of each of them with the boundary of the half-plane is a non-empty segment.} families of compact arc-connected sets have bounded chromatic number.
Suk \cite{Suk14} considered simple\footnote{A family of curves is \emph{simple} if any two of them intersect in at most one point.} families of $x$-monotone\footnote{A curve is \emph{$x$-monotone} if it meets every vertical line in at most one point.} curves grounded in a half-plane bounded by a vertical line; he proved that intersection graphs of such families satisfy $\chi=\smash[t]{2^{O(5^\omega)}}$.
Lasoń, Micek, Pawlik, and Walczak \cite{LMPW14} generalized both of these results, showing that intersection graphs of simple grounded families of compact arc-connected sets satisfy $\chi=\smash[t]{2^{O(2^\omega)}}$.
Our present result can be considered as a generalization of all the $\chi$-boundedness results mentioned in this paragraph to grounded families of curves with no restriction on the number of pairwise intersections.

Krawczyk and Walczak \cite{KW17} considered colorings of interval filament graphs (intersection graphs of continuous non-negative functions defined on closed intervals attaining zero values at their endpoints), which form a subclass of the class of outerstring graphs.
In particular, they constructed interval filament graphs with $\chi=\binom{\omega+1}{2}$.
This seems to be the best lower bound construction known to date for outerstring graphs.
Double outerstring graphs are perfect (satisfy $\chi=\omega$), which follows from their aforementioned characterization as incomparability graphs of partially ordered sets and from Dilworth's theorem \cite{Dil50}.

On the negative side, Burling \cite{Bur65} constructed triangle-free intersection graphs of axis-parallel boxes in $\setR^3$ with arbitrarily large chromatic number.
Using essentially the same construction, Pawlik et~al.\ \cite{PKK+13,PKK+14} proved existence of triangle-free intersection graphs of line segments (and various other kinds of geometric shapes) in the plane with arbitrarily large chromatic number.
These constructions show that some restriction on the layout of the geometric objects considered, like the one that the family is grounded, is indeed necessary to guarantee $\chi$-boundedness.

The best upper bound on the chromatic number of string graphs with $n$ vertices is $\smash[t]{(\log n)^{O(\log\omega)}}$, proved by Fox and Pach \cite{FP14} using a separator theorem due to Matoušek \cite{Mat14}, which is an earlier and weaker version of the aforementioned separator theorem of Lee \cite{Lee17}.
That was also the previous best upper bound for outerstring graphs (it follows from an earlier separator theorem for outerstring graphs \cite[Theorem 3.1]{FP12a} by the method described already in \cite{FP10}).
The above-mentioned result of Suk \cite{Suk14} implies that intersection graphs of simple families of $x$-monotone curves have chromatic number $O_\omega(\log n)$.\footnote{We write $O_\omega$ and $\Theta_\omega$ to denote asymptotic growth rate with respect to $n$ with $\omega$ fixed as a constant.}
On the other hand, the above-mentioned construction of Pawlik et~al.\ \cite{PKK+14} produces triangle-free segment intersection graphs with chromatic number $\Theta(\log\log n)$.
Krawczyk and Walczak \cite{KW17} further generalized it to a construction of string graphs with chromatic number $\Theta_\omega((\log\log n)^{\omega-1})$.
It is possible that all string graphs have chromatic number of order $\smash[t]{(\log\log n)^{f(\omega)}}$ for some function $f\colon\setN\to\setN$.
So far, bounds of this form have been established only for very special (still not $\chi$-bounded) classes of string graphs \cite{KPW15,KW17}.

Bounds on the chromatic number with respect to the clique number have also been studied for geometric disjointness graphs (complements of intersection graphs).
In particular, Pach and Tomon \cite{PT-arxiv} proved that disjointness graphs of $x$-monotone grounded curves satisfy $\chi\leq\binom{\omega+1}{2}$ (which they also proved to be tight) while disjointness graphs of $x$-monotone curves satisfy $\chi\leq\omega^2\binom{\omega+1}{2}$ (which improves the bound $\chi\leq\omega^4$ from \cite{LMPT94}).
By contrast, there are triangle-free disjointness graphs of grounded curves with arbitrarily large chromatic number \cite{MWW-arxiv}.

\subsection*{Quasi-planarity}

A \emph{topological graph} is a graph drawn in the plane so that each vertex is a point and each edge is a curve connecting the two endpoints of that edge and avoiding all other vertices.
Such a graph is \emph{$k$-quasi-planar} if it has no $k$ pairwise crossing edges (where a common endpoint is not considered as a crossing).
In particular, $2$-quasi-planar graphs are just planar graphs.
Any bound on the chromatic number of intersection graphs of curves with clique number less than $k$ implies a bound on the number of edges in $k$-quasi-planar graphs, as follows.
Given a $k$-quasi-planar topological graph $G$, if we shorten the edges a little at their endpoints so as to keep all crossings, then we obtain a family of curves with no clique of size $k$ in their intersection graph.
A proper coloring of these curves with $c$ colors yields an edge-decomposition of $G$ into $c$ planar graphs, whence the bound of $O(cn)$ on the number of edges of $G$ follows.

A well-known conjecture (see, e.g., \cite{PSS96} and \cite[Problem~1 in Section 9.6]{BMP-book}) asserts that $k$-quasi-planar topological graphs on $n$ vertices have $O_k(n)$ edges.
It was proved for $3$-quasi-planar simple\footnote{A topological graph is \emph{simple} if any two of its edges intersect in at most one point.} topological graphs by Agarwal et~al.\ \cite{AAP+97}, for all $3$-quasi-planar topological graphs by Pach, Radoičić, and Tóth \cite{PRT06}, and for $4$-quasi-planar topological graphs by Ackerman \cite{Ack09}.
Valtr \cite{Val97} proved the bound of $O_k(n\log n)$ on the number of edges in $k$-quasi-planar simple topological graphs with edges drawn as $x$-monotone curves.
An improvement of this result due to Fox, Pach, and Suk \cite{FPS13} provides the same conclusion without the simplicity assumption.
They also proved the bound of $\smash[t]{2^{\alpha(n)^c}}n\log n$ on the number of edges in $k$-quasi-planar simple topological graphs, where $\alpha$ denotes the inverse Ackermann function and $c$ depends only on $k$.
Suk and Walczak \cite{SW15} proved the same bound for $k$-quasi-planar topological graphs in which any two edges cross a bounded number of times, and they improved the bound for $k$-quasi-planar simple topological graphs to $O_k(n\log n)$.
The best known general upper bound on the number of edges in $k$-quasi-planar topological graphs, due to Fox and Pach \cite{FP12a,FP14}, is $\smash[t]{n(\log n)^{O(\log k)}}$.

\subsection*{Corollaries and follow-up generalizations}

An easy consequence of our present result is that the class of intersection graphs of curves each crossing a fixed straight line in exactly one point is $\chi$-bounded.
Indeed, we can color the parts of curves lying on each side of the line independently, and then we can color the entire curves using the pairs of colors obtained on the two sides.
This and a standard divide-and-conquer argument imply that intersection graphs of $x$-monotone curves have chromatic number $O_\omega(\log n)$, which yields an alternative proof of the result of Fox, Pach, and Suk \cite{FPS13} that $k$-quasi-planar topological graphs in which every edge is drawn as an $x$-monotone curve have $O_k(n\log n)$ edges.

By the same argument as is used in \cite{SW15} for simple families of curves, our present result has the following corollary: the class of intersection graphs of curves all of which intersect a fixed curve $c_0$ in exactly one point is $\chi$-bounded.
In the follow-up paper \cite{RW-inpress}, we prove the following generalization: for every integer $t\geq 1$, the class of intersection graphs of curves all of which intersect a fixed curve $c_0$ in at least one and at most $t$ points is $\chi$-bounded.
Then, in \cite{RW-inpress}, we use that generalization and a reduction from \cite{FPS13} to obtain the bound of $O_{k,t}(n\log n)$ on the number of edges of $k$-quasi-planar topological graphs in which any two edges cross at most $t$ times.
The aforementioned generalization depends on the result of the current paper, which serves as the base case for induction.
We refer the reader to \cite{RW-inpress} for more details.

\section{Preliminaries}

We let $H^+$ denote the upper closed half-plane determined by the horizontal axis.
We call the horizontal axis the \emph{baseline}.
We can assume, without loss of generality, that $H^+$ is the underlying half-plane of any grounded family of curves that we consider.
Accordingly, we call a family of curves \emph{grounded} if every curve in the family has one endpoint on the baseline and all the remaining part above the baseline, and we call a curve with this property itself \emph{grounded}.
The \emph{basepoint} of a grounded curve is the endpoint that lies on the baseline.
We can assume, without loss of generality, that the basepoints of all curves in any grounded family that we consider are distinct.
For if $b$ is the common basepoint of several curves in a grounded family $\famF$, then a sufficiently small neighborhood of $b$ is disjoint from the other curves in $\famF$ (because curves are compact sets and $\famF$ is finite), and an appropriate perturbation of the curves within that neighborhood of $b$ makes their basepoints distinct while keeping the curves pairwise intersecting.

A \emph{proper coloring} of a grounded family of curves $\famF$ is an assignment of colors to the curves in $\famF$ such that no two intersecting curves receive the same color.
A \emph{clique} in $\famF$ is a subfamily of $\famF$ comprising pairwise intersecting curves.
We let $\chi(\famF)$ and $\omega(\famF)$ denote the minimum number of colors in a proper coloring of $\famF$ and the maximum size of a clique in $\famF$, respectively.
Thus $\chi(\famF)$ is the chromatic number and $\omega(\famF)$ is the clique number of the intersection graph of $\famF$.

For any two grounded curves $c_1$ and $c_2$, we let $c_1\prec c_2$ denote that the basepoint of $c_1$ lies to the left of the basepoint of $c_2$.
In view of the assumption above, $\prec$ is a total order on any grounded family of curves that we consider---the left-to-right order of the basepoints.
The notation $\prec$ naturally extends to families of grounded curves: $\famF_1\prec\famF_2$ denotes that $c_1\prec c_2$ for all $c_1\in\famF_1$ and $c_2\in\famF_2$.
For any two grounded curves $c_1$ and $c_2$ with $c_1\prec c_2$ and any grounded family of curves $\famF$, we let $\famF(c_1,c_2)=\{c\in\famF\colon c_1\prec c\prec c_2\}$.

The following lemma is essentially due to McGuinness \cite[Lemma 2.1]{McG96}.
Here, we adapt it to our setting and include its short proof for completeness.

\begin{lemma}
\label{lem:mcguinness}
If\/ $\famF$ is a grounded family of curves with\/ $\chi(\famF)>2\alpha(\beta+1)$, where\/ $\alpha,\beta\geq 0$, then there is a subfamily\/ $\famH\subseteq\famF$ such that\/ $\chi(\famH)>\alpha$ and\/ $\chi(\famF(u,v))>\beta$ for any two intersecting curves\/ $u,v\in\famH$ with\/ $u\prec v$.
\end{lemma}

\begin{proof}
Partition $\famF$ into subfamilies $\famF_0\prec\cdots\prec\famF_n$ so that $\chi(\famF_i)=\beta+1$ for $0\leq i<n$ and $\chi(\famF_n)\leq\beta+1$.
This is done greedily, by processing the curves in $\famF$ in the order $\prec$, adding them to $\famF_0$ until $\chi(\famF_0)=\beta+1$, then adding them to $\famF_1$ until $\chi(\famF_1)=\beta+1$, and so on.
For $0\leq i\leq n$, a proper $(\beta+1)$-coloring of $\famF_i$ yields a partition of $\famF_i$ into color classes $\famF_i^1,\ldots,\famF_i^{\beta+1}$ each comprising pairwise disjoint curves.
Let $r\in\{1,\ldots,\beta+1\}$ be such that $\chi(\bigcup_{i=0}^n\famF_i^r)$ is maximized.
It follows that $\chi(\bigcup_{i=0}^n\famF_i^r)\geq\chi(\famF)/(\beta+1)>2\alpha$ and thus $\chi(\bigcup_{i\text{ even}}\famF_i^r)>\alpha$ or $\chi(\bigcup_{i\text{ odd}}\famF_i^r)>\alpha$.
Let $\famH=\bigcup_{i\text{ even}}\famF_i^r$ or $\famH=\bigcup_{i\text{ odd}}\famF_i^r$ accordingly, so that $\chi(\famH)>\alpha$.
Now, if two curves $u,v\in\famH$ with $u\prec v$ intersect, then $u\in\famF_k^r$ and $v\in\famF_\ell^r$ for two indices $k,\ell\in\{0,\ldots,n\}$ with $k<\ell$ (because no two curves in any $\famF_i^r$ intersect) of the same parity, and therefore there is at least one index $i\in\{k+1,\ldots,\ell-1\}$, implying $\famF_i\subseteq\famF(u,v)$ and thus $\chi(\famF(u,v))\geq\chi(\famF_i)>\beta$.
\end{proof}

A \emph{cap-curve} is a curve in $H^+$ that has both endpoints on the baseline and does not intersect the baseline in any other point.
It follows from the Jordan curve theorem that for every cap-curve $\gamma$, the set $H^+\setminus\gamma$ consists of two arc-connected components, one of which is bounded and denoted by $\int\gamma$ and the other is unbounded and denoted by $\ext\gamma$.
A point of the baseline belongs to $\int\gamma$ if and only if it lies strictly between the two endpoints of $\gamma$.

Most coloring arguments for geometric intersection graphs make essential use of the idea, originally due to Gyárfás \cite{Gya85}, to use distance levels to guarantee that each object to be colored lies within some restricted region and has a neighbor that crosses the boundary of that region.
The following lemma adapts this idea to our setting.
See Figure~\ref{fig:bfs} for an illustration.

\begin{figure}[t]
\centering
\begin{tikzpicture}[yscale=.9]
  \fill[black!15] plot[smooth,tension=.7] coordinates {(4.3,0) (4.5,2.5) (6.3,4.3) (8.7,4.1) (10.2,2.4) (10.6,0)}--cycle;
  \draw[dotted] plot[smooth,tension=.7] coordinates {(4.3,0) (4.5,2.5) (6.3,4.3) (8.7,4.1) (10.2,2.4) (10.6,0)};
  \draw[blue] plot[smooth,tension=.7] coordinates {(0.9,0) (1.1,1.8) (2.3,2.6) (3.8,1.7)};
  \draw[blue] plot[smooth,tension=.65] coordinates {(1.8,0) (1.6,1.4) (1.2,3.4) (2.1,4.2) (4.3,3.2)};
  \draw[blue] plot[smooth,tension=.7] coordinates {(3.3,0) (3.2,1.4) (3.7,2.6)};
  \draw[blue] plot[smooth,tension=.68] coordinates {(11.4,0) (11.05,2.6) (9.5,4.5) (6.8,5.2) (3.9,4.8) (2.3,3.5) (2.3,1.9)};
  \draw[red] plot[smooth,tension=.7] coordinates {(2.5,0) (2.8,1.3) (4.1,0.8) (5.5,1)};
  \draw[red] plot[smooth,tension=.7] coordinates {(6.6,0) (6.3,2.7) (4.9,4.3) (3.2,3.2)};
  \draw[red] plot[smooth,tension=.7] coordinates {(9.2,0) (9.2,2.2) (9.5,3.6) (10.1,4.8) (9.4,5.2) (8.3,4.5) (7.7,3.4) (7.7,2.4)};
  \draw plot[smooth,tension=.7] coordinates {(5.1,0) (5,1.3) (5.2,2.4) (5.7,2.5) (6,0.8)};
  \draw plot[smooth,tension=.7] coordinates {(7.3,0) (7.2,2.8) (8.9,3.1)};
  \draw plot[smooth,tension=.7] coordinates {(7.9,0) (7.3,1.6) (5.5,1.7)};
  \draw[dashed] plot[smooth,tension=.7] coordinates {(8.5,0) (8.5,1.8) (8.4,3.65)};
  \draw plot[smooth,tension=.7] coordinates {(9.9,0) (9.5,1.6) (8,1.7)};
  \draw (0,0)--(12.3,0);
  \node[below] at (0.9,0) {$c_0$};
  \node[below] at (5.1,0) {$c_1$};
  \node[below] at (7.3,0) {$c_2$};
  \node[below] at (7.9,0) {$c_3$};
  \node[below] at (9.9,0) {$c_4$};
  \node[above] at (6.3,4.3) {$\gamma$};
  \node at (6.9,3.7) {$\int\gamma$};
  \node at (2.6,4.6) {$E$};
\end{tikzpicture}
\caption{Illustration for Lemma~\ref{lem:bfs}: $\famG=\{c_1,c_2,c_3,c_4\}$}
\label{fig:bfs}
\end{figure}

\begin{lemma}
\label{lem:bfs}
For every grounded family of curves\/ $\famF$ with\/ $\omega(\famF)\geq 2$, there are a cap-curve\/ $\gamma$ and a subfamily\/ $\famG\subseteq\famF$ with\/ $\chi(\famG)\geq\chi(\famF)/2$ such that every curve in\/ $\famG$ is entirely contained in\/ $\int\gamma$ and intersects some curve in\/ $\famF$ that intersects\/ $\gamma$.
\end{lemma}

\begin{proof}
Since the chromatic number of a graph is the maximum of the chromatic numbers of its connected components, we can assume without loss of generality that the intersection graph of $\famF$ is connected (otherwise we can restrict $\famF$ to the component with maximum chromatic number).

Let $c_0$ be the curve in $\famF$ with leftmost basepoint.
For $i\geq 0$, let $\famF_i$ denote the family of curves in $\famF$ that are at distance $i$ from $c_0$ in the intersection graph of $\famF$.
It follows that $\famF_0=\{c_0\}$ and every curve in $\famF_i$ is disjoint from every curve in $\famF_j$ whenever $\lvert i-j\rvert\geq 2$.
Since $\bigcup_{i=0}^\infty\famF_i=\famF$, we have $\chi(\bigcup_{i\text{ even}}\famF_i)\geq\chi(\famF)/2$ or $\chi(\bigcup_{i\text{ odd}}\famF_i)\geq\chi(\famF)/2$, and therefore there is an index $d\geq 1$ such that $\chi(\famF_d)\geq\chi(\famF)/2$.
There is a subfamily $\famG\subseteq\famF_d$ such that $\chi(\famG)=\chi(\famF_d)\geq\chi(\famF)/2$ and the intersection graph of $\famG$ is connected.
The latter implies that the union of the curves in $\famG$ contains a cap-curve $\nu$ connecting the leftmost and the rightmost basepoints of the curves in $\famG$.

Let $b_0$ be the basepoint of $c_0$.
Let $E=\{b_0\}$ when $d=1$ and $E$ be the union of the curves in $\bigcup_{i=0}^{d-2}\famF_i$ when $d\geq 2$ (in particular, $b_0\in E$).
Let $G$ be the union of the curves in $\famG$.
Since every curve in $\bigcup_{i=0}^{d-2}\famF_i$ is disjoint from every curve in $\famG$ and the intersection graphs of both families are connected, the sets $E$ and $G$ are disjoint arc-connected subsets of $H^+$.
Furthermore, the set $E$ lies in the unbounded arc-connected component of $H^+\setminus G$, because so does the point $b_0$.
Therefore, there is a cap-curve $\gamma$ separating $E$ and $G$ in $H^+$ so that $G\subset\int\gamma$ and $E\subset\ext\gamma$.

For every curve $c\in\famG$, since $c\in\famF_d$, there is a curve $s\in\famF_{d-1}$ that intersects $c$.
Since $s\in\famF_{d-1}$, either $s=c_0$ (when $d=1$) or $s$ intersects some curve in $\famF_{d-2}$ (when $d\geq 2$).
In either case, $s$ intersects both $E$ and $G$.
This, $E\subset\ext\gamma$, and $G\subset\int\gamma$ imply that $s$ intersects $\gamma$.
We conclude that every curve in $\famG$ intersects some curve in $\famF$ that intersects $\gamma$.
\end{proof}

The next lemma is essentially a reformulation of the known fact (mentioned in the introduction) that double outerstring graphs are perfect.
We provide its proof for completeness.

\begin{lemma}
\label{lem:perfect}
If\/ $\gamma$ is a cap-curve and\/ $\famU$ is a grounded family of curves each having one endpoint (other than the basepoint) on\/ $\gamma$ and all the remaining part in\/ $\int\gamma$, then\/ $\chi(\famU)=\omega(\famU)$.
\end{lemma}

\begin{proof}
Consider a binary relation $<$ on $\famU$ defined as follows, for any $u_1,u_2\in\famU$:
\begin{equation*}
u_1<u_2\quad\text{if and only if}\quad u_1\prec u_2\quad\text{and}\quad u_1\cap u_2=\emptyset.
\end{equation*}
We claim that $<$ is a strict partial order on $\famU$.
Clearly, it is an irreflexive and antisymmetric relation.
For transitivity, suppose $u_1<u_2$ and $u_2<u_3$ but not $u_1<u_3$, for some $u_1,u_2,u_3\in\famU$.
Since $u_1\prec u_2\prec u_3$, it follows that $u_1\cap u_3\neq\emptyset$.
Therefore, there is a cap-curve $\nu\subseteq u_1\cup u_3$ that connects the basepoints of $u_1$ and $u_3$.
The basepoint of $u_2$ lies in $\int\nu$, while the other endpoint of $u_2$ lies in $\ext\nu$, as it lies on $\gamma$ and $\nu\subset\int\gamma$.
It follows that $u_2$ intersects $\nu$.
This yields $u_1\cap u_2\neq\emptyset$ or $u_2\cap u_3\neq\emptyset$, either of which is a contradiction.

Since the antichains of the order $<$ on $\famU$ are the cliques in the intersection graph of $\famU$, the width (maximum size of an antichain) of $<$ is $\omega(\famU)$.
Therefore, by Dilworth's theorem \cite{Dil50}, the family $\famU$ can be partitioned into $\omega(\famU)$ chains with respect to $<$.
Such a partition is equivalent to a proper coloring of $\famU$ with $\omega(\famU)$ colors.
\end{proof}

\section{Proof}

\subsection*{Setup}

Here is our main theorem in the form that we are going to prove.

\begin{mainthm}[rephrased]
There is a function\/ $f\colon\setN\to\setN$ with\/ $f(k)=\smash[t]{2^{O(2^{k(k-1)/2})}}$ such that for every\/ $k\in\setN$, every grounded family of curves\/ $\famF$ with\/ $\omega(\famF)\leq k$ satisfies\/ $\chi(\famF)\leq f(k)$.
\end{mainthm}

The proof proceeds by induction on $k$.
We let $f(1)=1$, which clearly satisfies the conclusion of the theorem for $k=1$.
For $k\geq 2$, we assume that $f(k-1)$ is an upper bound on the chromatic number of grounded families of curves with clique number at most $k-1$, and we use it to derive an upper bound $f(k)$ on the chromatic number of grounded families of curves with clique number at most $k$.
The induction hypothesis is applied only as follows: if $\famF$ is a grounded family of curves with $\omega(\famF)\leq k$ and $c\in\famF$, then the family of curves in $\famF\setminus\{c\}$ that intersect $c$ (that is, the neighborhood of $c$ in the intersection graph of $\famF$) has clique number at most $k-1$ and therefore has chromatic number at most $f(k-1)$.

This motivates the following definition: for $\xi\in\setN$, a \emph{$\xi$-family} is a grounded family of curves $\famF$ such that for every curve $c\in\famF$, the family of curves in $\famF\setminus\{c\}$ that intersect $c$ has chromatic number at most $\xi$.
In the rest of the paper, we prove the following lemma.

\begin{lemma}
\label{lem:main}
There is a constant\/ $\zeta=\smash[t]{2^{O(k2^{2k})}\xi^{2^{k-1}-1}}$ such that every\/ $\xi$-family\/ $\famF$ with\/ $\omega(\famF)\leq k$ satisfies\/ $\chi(\famF)\leq\zeta$.
\end{lemma}

\noindent
Hereafter, for clarity of presentation, we treat $k$ and $\xi$ as fixed integer constants (except in statements on asymptotic growth rate) with $k\geq 2$ and $\xi\geq 1$.

To complete the induction step in the proof of our main theorem, we fix $\xi=f(k-1)$, and we let $f(k)$ be the constant $\zeta$ claimed by Lemma~\ref{lem:main} for $k$ and $\xi$.
By the induction hypothesis, every grounded family of curves with clique number at most $k$ is an $f(k-1)$-family and therefore, by Lemma~\ref{lem:main}, has chromatic number at most $f(k)$.

To see that the bound on $\zeta$ from Lemma~\ref{lem:main} implies the bound on $f$ from our main theorem, suppose $k\geq 2$ and $\log_2f(k-1)\leq C\cdot\smash[t]{2^{(k-1)(k-2)/2}}$ for some large constant $C$.
Lemma~\ref{lem:main} applied with $\xi=f(k-1)$ then provides the following bound:
\begin{equation*}
\begin{split}
\log_2\chi(\famF)&\leq C\cdot\smash[t]{2^{(k-1)(k-2)/2}}\cdot(2^{k-1}-1)+O(k2^{2k})\\
&=C\cdot\smash[t]{2^{k(k-1)/2}}-C\cdot\smash[t]{2^{(k-1)(k-2)/2}}+O(k2^{2k})\leq C\cdot\smash[t]{2^{k(k-1)/2}},
\end{split}
\end{equation*}
where the last inequality holds when $C$ is large enough.
This completes the proof of our main theorem provided that we have Lemma~\ref{lem:main}.
It remains to prove the lemma.

Here is an overview of the proof of Lemma~\ref{lem:main}.
Assuming that $\famF$ is a $\xi$-family with clique number at most $k$ and sufficiently large chromatic number, we show that some specific configurations must occur in $\famF$.
Our goal is to reach a contradiction by finding a clique of size $k+1$ in $\famF$.

First, we show that every $\xi$-family with clique number at most $k$ and sufficiently large chromatic number contains a subfamily with large chromatic number \emph{supported by a skeleton} or \emph{supported from outside} (Lemma~\ref{lem:sideways}).
Iterating this step on the successive subfamilies, we construct a long chain of subfamilies supported by skeletons or supported from outside.
Then, the proof splits into two parts.
In the first part, we show that a long chain of subfamilies supported from outside contains a long \emph{bracket system} (proof of Lemma~\ref{lem:skeleton}), which then contains a large clique (Lemma~\ref{lem:bracket-system}).
In the second part, we show that a long chain of subfamilies supported by skeletons contains a long \emph{tree-configuration} (Lemma~\ref{lem:tree-configuration}), which then contains a large clique (Lemma~\ref{lem:tree-configuration-clique}).

\subsection*{Skeletons}

A \emph{skeleton} is a pair $(\gamma,\famU)$ such that $\gamma$ is a cap-curve and $\famU$ is a family of pairwise disjoint grounded curves each having one endpoint (other than the basepoint) on $\gamma$ and all the remaining part in $\int\gamma$ (see Figure~\ref{fig:skeleton}).
For a grounded family of curves $\famF$, a skeleton $(\gamma,\famU)$ is an \emph{$\famF$-skeleton} if every curve in $\famU$ is a subcurve of some curve in $\famF$.
A grounded family of curves $\famG$ is \emph{supported} by a skeleton $(\gamma,\famU)$ if every curve in $\famG$ lies entirely in $\int\gamma$ and intersects some curve in $\famU$.
A subfamily $\famH$ of a grounded family of curves $\famF$ is \emph{supported from outside} in $\famF$ if every curve in $\famH$ intersects some curve in the set $\{s\in\famF\colon s\prec\famH$ or $\famH\prec s\}$.

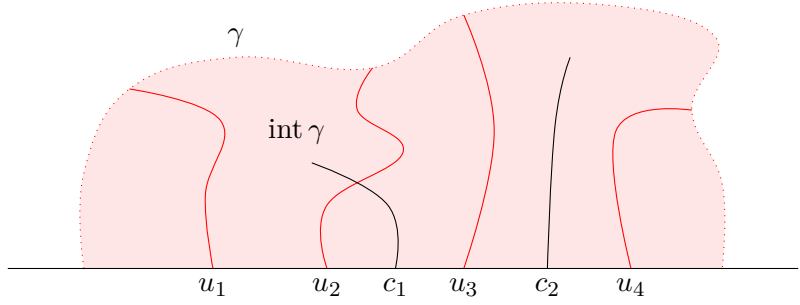
\begin{figure}[t]
\centering
\begin{tikzpicture}[yscale=.7]
  \draw[name path=g,red,dotted,fill=red!10] plot[smooth,tension=.7] coordinates {(1,0) (1,1.8) (1.6,3.4) (3,4) (4.8,3.8) (6,4.8) (7.8,5) (9.3,4.4) (9,3) (9.4,1.6) (9.4,0)};
  \draw[red] plot[smooth,tension=.7] coordinates {(2.7,0) (2.6,1.4) (2.8,2.8) (1.6,3.4)};
  \draw[red] plot[smooth,tension=.7] coordinates {(4.2,0) (4.2,1.2) (5.2,2.2) (4.6,3) (4.8,3.8)};
  \draw[red] plot[smooth,tension=.7] coordinates {(6,0) (6.4,2.6) (6,4.8)};
  \draw[red] plot[smooth,tension=.7] coordinates {(8.2,0) (8,2.6) (9,3)};
  \draw (0,0)--(10.4,0);
  \draw plot[smooth,tension=.75] coordinates {(5.1,0) (5,1.2) (4,2)};
  \draw plot[smooth,tension=.75] coordinates {(7.1,0) (7.2,2.7) (7.4,4)};
  \node[above] at (3,4) {$\gamma$};
  \node[below] at (2.7,0) {$u_1$};
  \node[below] at (4.2,0) {$u_2$};
  \node[below] at (6,0) {$u_3$};
  \node[below] at (8.2,0) {$u_4$};
  \node[below] at (5.1,0) {$c_1$};
  \node[below] at (7.1,0) {$c_2$};
  \node at (3.8,2.6) {$\int\gamma$};
\end{tikzpicture}
\caption{A skeleton $\bigl(\gamma,\{u_1,u_2,u_3,u_4\}\bigr)$, which supports $c_1$ but not $c_2$}
\label{fig:skeleton}
\end{figure}

\begin{lemma}
\label{lem:sideways}
For any\/ $\alpha,\beta\in\setN$, every\/ $\xi$-family\/ $\famF$ with\/ $\omega(\famF)\leq k$ and\/ $\chi(\famF)>2(k\alpha+\beta)$ contains at least one of the following configurations:
\begin{itemize}
\item a subfamily\/ $\famG\subseteq\famF$ with\/ $\chi(\famG)>\alpha$ supported by an\/ $\famF$-skeleton,
\item a subfamily\/ $\famH\subseteq\famF$ with\/ $\chi(\famH)>\beta$ supported from outside in\/ $\famF$.
\end{itemize}
\end{lemma}

\begin{proof}
Apply Lemma~\ref{lem:bfs} to obtain a cap-curve $\gamma$ and a subfamily $\famG\subseteq\famF$ with $\chi(\famG)>k\alpha+\beta$ such that every curve in $\famG$ lies entirely in $\int\gamma$ and intersects some curve in $\famS$, where $\famS$ are the curves in $\famF$ that intersect $\gamma$.
Let $\famU$ be the family of grounded curves obtained by taking, for each curve $s\in\famS$ with basepoint in $\int\gamma$, the part of $s$ from the basepoint to the first intersection point with~$\gamma$.
The family $\famU$ satisfies the assumption of Lemma~\ref{lem:perfect}, which yields $\chi(\famU)=\omega(\famU)\leq\omega(\famF)\leq k$, as $\famF$ is a $\xi$-family.
A proper $k$-coloring partitions $\famU$ into families $\famU_1,\ldots,\famU_k$ each consisting of pairwise disjoint curves.
It follows that $(\gamma,\famU_1),\ldots,(\gamma,\famU_k)$ are $\famF$-skeletons.

For $1\leq i\leq k$, let $\famG_i$ be the curves in $\famG$ that intersect some curve in $\famU_i$, so that the family $\famG_i$ is supported by the skeleton $(\gamma,\famU_i)$.
If $\chi(\famG_i)>\alpha$, then $\famG_i$ and $(\gamma,\famU_i)$ satisfy the first condition in the conclusion of the lemma, so suppose $\chi(\famG_i)\leq\alpha$, for $1\leq i\leq k$.
Let $\famG'=\famG\setminus(\famG_1\cup\cdots\cup\famG_k)$.
It follows that $\chi(\famG')\geq\chi(\famG)-k\alpha>\beta$.

Since the chromatic number of a graph is the maximum of the chromatic numbers of its connected components, there is a subfamily $\famH\subseteq\famG'$ such that $\chi(\famH)=\chi(\famG')>\beta$ and the intersection graph of $\famH$ is connected.
The latter implies that the union of the curves in $\famH$ contains a cap-curve $\nu$ connecting the leftmost and the rightmost basepoints of the curves in $\famH$.
Suppose there is a curve $s\in\famS$ with basepoint between the leftmost and the rightmost basepoints of the curves in $\famH$.
The part of $s$ from the basepoint to the first intersection point with $\gamma$ is a curve in some $\famU_i$ ($1\leq i\leq k$) that must intersect $\nu$ (as $\nu\subset\int\gamma$) and thus some curve in $\famG'$ (as $\nu$ lies in the union of the curves in $\famG'$).
This yields $\famG'\cap\famG_i\neq\emptyset$, which is a contradiction.
Thus $s\prec\famH$ or $\famH\prec s$.
We conclude that $\famH$ satisfies the second condition in the conclusion of the lemma.
\end{proof}

\subsection*{Brackets and bracket systems}

A \emph{bracket} is a pair $(\famH,\famS)$ of families of grounded curves with the following properties:
\begin{itemize}
\item $\famS\prec\famH$ or $\famH\prec\famS$,
\item every curve in $\famH$ intersects some curve in $\famS$.
\end{itemize}
To motivate this definition, observe that a subfamily $\famH$ supported from outside in a grounded family of curves $\famF$ gives rise to two brackets $(\famH^L,\famS^L)$ and $(\famH^R,\famS^R)$ with $\famH=\famH^L\cup\famH^R$, where
\begin{alignat*}{2}
\famS^L&=\{s\in\famF\colon s\prec\famH\},\qquad & \famH^L&=\{c\in\famH\colon c\text{ intersects a curve in }\famS^L\},\\
\famS^R&=\{s\in\famF\colon\famH\prec s\},\qquad & \famH^R&=\{c\in\famH\colon c\text{ intersects a curve in }\famS^R\}.
\end{alignat*}
For a fixed bracket $(\famH,\famS)$ and a curve $c\in\famH$, we introduce the following notation (see Figure~\ref{fig:bracket}):
\begin{itemize}
\item $p(c)$ is the first intersection point of $c$ with a curve in $\famS$ encountered when following $c$ in the direction from the basepoint towards the other endpoint,
\item $s(c)$ is an arbitrarily chosen curve in $\famS$ that contains the point $p(c)$,
\item $c'$ is the part of $c$ from the basepoint to $p(c)$ excluding the point $p(c)$,
\item $\nu(c)$ is the cap-curve formed by the union of $c'$ and the part of $s(c)$ from the basepoint to $p(c)$,
\item $I(c)$ is the closed bounded region determined by $\nu(c)$ and the part of the baseline between the two endpoints of $\nu(c)$.
\end{itemize}
It follows that $c'$ is disjoint from every curve in $\famS$, for every $c\in\famH$.
Finally, we let $I=\bigcap_{c\in\famH}I(c)$, and we call $I$ the \emph{internal region} of the bracket $(\famH,\famS)$.

\begin{figure}[t]
\centering
\begin{tikzpicture}
  \path[name path=s1] plot[smooth,tension=.7] coordinates {(0.5,0) (2,4.5) (5.5,2.5) (7.5,2.5)};
  \path[name path=s2] plot[smooth,tension=.7] coordinates {(2,0) (2.5,5) (7,5) (9.5,5.5)};
  \path[name path=s3] plot[smooth,tension=.7] coordinates {(3,0) (3,4) (8,4)};
  \path[name path=c1] plot[smooth,tension=.7] coordinates {(5.5,0) (5.5,2.5) (7.5,3.5) (10,2.8)};
  \path[name path=c2] plot[smooth,tension=.7] coordinates {(6.5,0) (7,1.5) (9.5,3) (7.6,5.9)};
  \path[name path=c3] plot[smooth,tension=.7] coordinates {(7.5,0) (6.5,3.5) (3,3) (0.5,4.3)};
  \path[name path=c4] plot[smooth,tension=.7] coordinates {(9,0) (8,3) (6,5.5)};
  \begin{scope}
  \clip plot[smooth,tension=.7] coordinates {(0.5,0) (2,4.5) (5.5,2.5) (7.5,2.5)}--(7.5,0)--cycle;
  \clip plot[smooth,tension=.7] coordinates {(3,0) (3,4) (8,4)}--(8,0)--cycle;
  \fill[black!15] plot[smooth,tension=.7] coordinates {(5.5,0) (5.5,2.5) (7.5,3.5) (10,3)}--(10,5)--(0.5,5)--(0.5,0)--cycle;
  \end{scope}
  \draw[red] plot[smooth,tension=.7] coordinates {(0.5,0) (2,4.5) (5.5,2.5) (7.5,2.5)};
  \draw[red] plot[smooth,tension=.7] coordinates {(2,0) (2.5,5) (7,5) (9.5,5.5)};
  \draw[red] plot[smooth,tension=.7] coordinates {(3,0) (3,4) (8,4)};
  \draw[intersection segments={of=c1 and s1,sequence=L1}];
  \draw[intersection segments={of=c2 and s2,sequence=L1}];
  \draw[intersection segments={of=c3 and s1,sequence=L1}];
  \draw[intersection segments={of=c4 and s3,sequence=L1}];
  \draw[dashed,intersection segments={of=c1 and s1,sequence=L2}];
  \draw[dashed,intersection segments={of=c2 and s2,sequence=L2}];
  \draw[dashed,intersection segments={of=c3 and s1,sequence=L2--L3--L4}];
  \draw[dashed,intersection segments={of=c4 and s3,sequence=L2}];
  \draw[fill=white,name intersections={of=c1 and s1}] (intersection-1) circle (2pt);
  \draw[fill=white,name intersections={of=c2 and s2}] (intersection-1) circle (2pt);
  \draw[fill=white,name intersections={of=c3 and s1}] (intersection-1) circle (2pt);
  \draw[fill=white,name intersections={of=c4 and s3}] (intersection-1) circle (2pt);
  \draw (-0.5,0)--(10.5,0);
  \node[below] at (0.5,0) {$s_1$};
  \node[below] at (2,0) {$s_2$};
  \node[below] at (3,0) {$s_3$};
  \node[below] at (5.5,0) {$c_1$};
  \node[below] at (6.5,0) {$c_2$};
  \node[below] at (7.5,0) {$c_3$};
  \node[below] at (9,0) {$c_4$};
  \node at (4,2) {$I$};
\end{tikzpicture}
\caption{A bracket $(\{c_1,c_2,c_3,c_4\},\{s_1,s_2,s_3\})$ with internal region $I$ and with $s_1=s(c_1)=s(c_3)$, $s_2=s(c_2)$, $s_3=s(c_4)$.
The part $c_i'$ of each $c_i$ is drawn solid.}
\label{fig:bracket}
\end{figure}
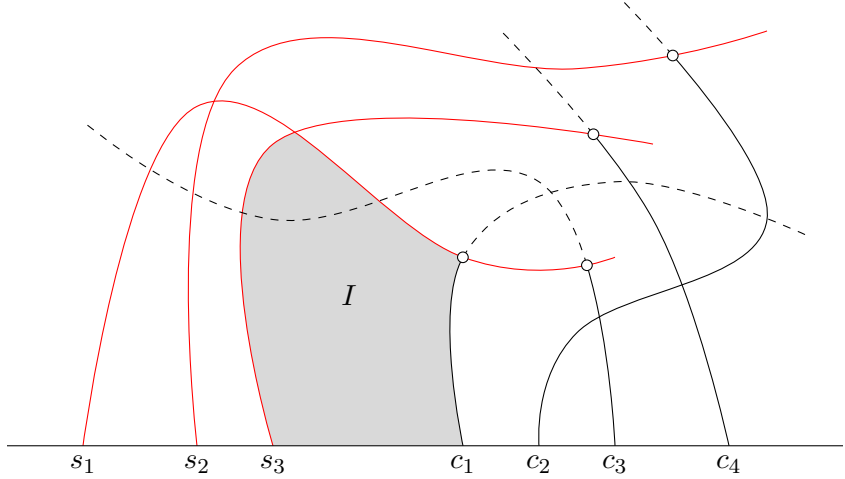

\begin{lemma}
\label{lem:boundary}
If\/ $(\famH,\famS)$ is a bracket with internal region\/ $I$, then there are two cliques\/ $\famH_I\subseteq\famH$ and\/ $\famS_I\subseteq\famS$ such that every curve with basepoint in\/ $I$ lies entirely in\/ $I$ or intersects at least one of the curves in\/ $\famH_I\cup\famS_I$.
\end{lemma}

\begin{proof}
Assume $\famS\prec\famH$; the case that $\famH\prec\famS$ is analogous.
Let $\famH_I=\{c\in\famH\colon c'\cap I\neq\emptyset\}$ and $\famS_I=\{s(c)\in\famS\colon c\in\famH$ and $s(c)\cap I\neq\emptyset\}$.
To see that $\famH_I$ is a clique, suppose there are two disjoint curves $c_1,c_2\in\famH_I$ with $c_1\prec c_2$.
It follows that $c_2'$ is disjoint from $c_1'$ and $s(c_1)$.
This and $s(c_1)\prec c_1\prec c_2$ imply that $c_2'$ is disjoint from $I(c_1)$.
This and $I\subseteq I(c_1)$ contradict the assumption that $c_2'\cap I\neq\emptyset$.
To see that $\famS_I$ is a clique, suppose there are $c_1,c_2\in\famH$ such that $s(c_1)$ and $s(c_2)$ are two disjoint curves in $\famS_I$ with $s(c_1)\prec s(c_2)$.
This, the fact that $s(c_1)$ is also disjoint from $c_2'$, and $s(c_1)\prec s(c_2)\prec c_2$ imply that $s(c_1)$ is disjoint from $I(c_2)$.
This and $I\subseteq I(c_2)$ contradict the assumption that $s(c_1)\cap I\neq\emptyset$.

Let $z$ be a curve with basepoint in $I$ that does not lie entirely in $I$.
It follows that $z\not\subset I(c)$ and therefore $z$ intersects $\nu(c)$ for at least one curve $c\in\famH$.
Let $c\in\famH$ be such that $\nu(c)$ is the first curve of this form intersected by $z$, and let $p$ be the first intersection point of $z$ with $\nu(c)$.
This yields $p\in I$, as $I$ is a closed set.
The fact that $p\in\nu(c)\subseteq c'\cup s(c)$ implies $p\in c'$, in which case $c\in\famH_I$, or $p\in s(c)$, in which case $s(c)\in\famS_I$.
In either case, $z$ intersects a curve in $\famH_I\cup\famS_I$.
\end{proof}

A \emph{bracket system} is a sequence of brackets $\bigl((\famH_0,\famS_0),\ldots,(\famH_n,\famS_n)\bigr)$ with internal regions $I_0,\ldots,I_n$, respectively, and with the following properties, for every $i$ with $0\leq i\leq n-1$:
\begin{itemize}
\item every curve in $\famH_i$ lies entirely in $I_{i+1}\cap\cdots\cap I_n$,
\item $\famS_i\prec(\famH_{i+1}\cup\famS_{i+1})\cup\cdots\cup(\famH_n\cup\famS_n)$ or $(\famH_{i+1}\cup\famS_{i+1})\cup\cdots\cup(\famH_n\cup\famS_n)\prec\famS_i$.
\end{itemize}

\begin{lemma}
\label{lem:bracket-system}
Let\/ $\bigl((\famH_0,\famS_0),\ldots,(\famH_n,\famS_n)\bigr)$ be a bracket system in a\/ $\xi$-family\/ $\famF$.
If\/ $\chi(\famH_i)>i\xi$ for every\/ $i$ with\/ $0\leq i\leq n$, then there is a clique\/ $\{s_0,\ldots,s_n\}$ with\/ $s_i\in\famS_i$ for every\/ $i$ with\/ $0\leq i\leq n$.
\end{lemma}

\begin{proof}
We proceed by induction on $n$.
The assumption that $\chi(\famH_0)>0$ yields $\famH_0\neq\emptyset$ and thus $\famS_0\neq\emptyset$.
Choose any $s_0\in\famS_0$.
This already completes the proof for the base case of $n=0$, as $\{s_0\}$ is a clique.
Therefore, for the rest of the proof, suppose that $n\geq 1$ and the lemma holds for $n-1$.
By the first property of a bracket system, since $s_0$ intersects a curve in $\famH_0$, $s_0$ intersects $I_1\cap\cdots\cap I_n$.
This and the definition of the internal region of a bracket imply that $s_0$ intersects $I(c)$ for all $c\in\famH_1\cup\cdots\cup\famH_n$.
This and the second property of a bracket system imply that $s_0$ intersects $\nu(c)$ for all $c\in\famH_1\cup\cdots\cup\famH_n$.
For $1\leq i\leq n$, let $\famH_i'$ be the curves in $\famH_i$ that do not intersect $s_0$, and let $\famS_i'=\{s(c)\colon c\in\famH_i'\}$.
It follows that $\chi(\famH_i\setminus\famH_i')\leq\xi$ (as $\famF$ is a $\xi$-family) and thus $\chi(\famH_i')\geq\chi(\famH_i)-\xi>(i-1)\xi$ for $1\leq i\leq n$.
Therefore, we can apply the induction hypothesis to the bracket system $\bigl((\famH_1',\famS_1'),\ldots,(\famH_n',\famS_n')\bigr)$ to find a clique $\{s_1,\ldots,s_n\}$ with $s_i\in\famS_i'$ for $1\leq i\leq n$.
For every $c\in\famH_1'\cup\cdots\cup\famH_n'$, since $s_0$ intersects $\nu(c)$ but not $c$, it intersects $s(c)$.
In particular, $s_0$ intersects each of $s_1,\ldots,s_n$, and therefore $\{s_0,\ldots,s_n\}$ is a clique.
\end{proof}

\begin{lemma}
\label{lem:skeleton}
There is a function\/ $g\colon\setN\to\setN$ with\/ $g(\alpha)=\smash[t]{2^{O(k)}}(\xi+\alpha)$ such that for every\/ $\alpha\in\setN$, every\/ $\xi$-family\/ $\famF$ with\/ $\omega(\famF)\leq k$ and\/ $\chi(\famF)>g(\alpha)$ contains a subfamily\/ $\famG\subseteq\famF$ with\/ $\chi(\famG)>\alpha$ supported by an\/ $\famF$-skeleton.
\end{lemma}

\begin{proof}
Fix $\alpha\in\setN$, and let 
\begin{alignat*}{3}
\gamma_1&=0,\qquad & \gamma_{i+1}&=2(\gamma_i+2k\xi+i\xi+1)\quad & \text{for }1\leq i\leq k,\\
\beta_{k+1}&=\gamma_{k+1},\qquad & \beta_i&=2(k\alpha+\beta_{i+1})\quad & \text{for }0\leq i\leq k.
\end{alignat*}
Let $g(\alpha)=\beta_0$.
It easily follows from the definition above that $g$ has the required order of magnitude.
It remains to prove that $g$ has the property claimed by the lemma.

Let $\famF$ be a $\xi$-family with $\omega(\famF)\leq k$ and $\chi(\famF)>\beta_0$.
Suppose for the sake of contradiction that every subfamily of $\famF$ supported by an $\famF$-skeleton has chromatic number at most $\alpha$.
Let $\famF_0=\famF$.
Apply Lemma~\ref{lem:sideways} (and the second conclusion thereof) $k+1$ times to find families $\famF_1,\ldots,\famF_{k+1}$ with the following properties:
\begin{itemize}
\item $\famF=\famF_0\supseteq\famF_1\supseteq\cdots\supseteq\famF_{k+1}$,
\item for $0\leq i\leq k$, the family $\famF_{i+1}$ is supported from outside in $\famF_i$,
\item for $0\leq i\leq k+1$, we have $\chi(\famF_i)>\beta_i$.
\end{itemize}
In particular, by the last condition, we have $\chi(\famF_{k+1})>\beta_{k+1}=\gamma_{k+1}$.

We claim that there are families $\famG_1,\ldots,\famG_k$ and brackets $(\famH_0,\famS_0),\ldots,(\famH_k,\famS_k)$ with internal regions $I_0,\ldots,I_k$, respectively, and with the following properties:
\begin{itemize}
\item $\famG_1\subseteq\cdots\subseteq\famG_k\subseteq\famG_{k+1}=\famF_{k+1}$,
\item for $0\leq i\leq k$, we have $\famH_i\subseteq\famG_{i+1}$ and $\chi(\famH_i)=i\xi+1$,
\item for $0\leq i\leq k$, we have $\famS_i\subseteq\famF_i$ and either $\famS_i\prec\famF_{i+1}$ or $\famF_{i+1}\prec\famS_i$,
\item for $1\leq i\leq k$, every curve in $\famG_i$ is entirely contained in $I_i\cap\cdots\cap I_k$,
\item for $1\leq i\leq k+1$, we have $\chi(\famG_i)>\gamma_i$.
\end{itemize}
This suffices for the proof of the lemma---the properties above imply that $\bigl((\famH_0,\famS_0),\ldots,(\famH_k,\famS_k)\bigr)$ is a bracket system of length $k+1$ that satisfies the assumption of Lemma~\ref{lem:bracket-system} and therefore (by Lemma~\ref{lem:bracket-system}) contains a clique of size $k+1$, contradicting the assumption that $\omega(\famF)\leq k$.

Let $\famG_{k+1}=\famF_{k+1}$.
For $1\leq i\leq k$ (considered in the order from $k$ to $1$), we assume that we have already found $\famG_{i+1}$, and we show how to find $\famH_i$, $\famS_i$, and $\famG_i$.
Let
\begin{alignat*}{2}
\famF_i^L&=\{s\in\famF_i\colon s\prec\famF_{i+1}\},\qquad & \famG_{i+1}^L&=\{c\in\famG_{i+1}\colon c\text{ intersects a curve in }\famF_i^L\},\\
\famF_i^R&=\{s\in\famF_i\colon\famF_{i+1}\prec s\},\qquad & \famG_{i+1}^R&=\{c\in\famG_{i+1}\colon c\text{ intersects a curve in }\famF_i^R\}.
\end{alignat*}
Since $\famG_{i+1}\subseteq\famF_{i+1}$ and $\famF_{i+1}$ is supported from outside in $\famF_i$, we have $\famG_{i+1}=\famG_{i+1}^L\cup\famG_{i+1}^R$ and thus $\chi(\famG_{i+1}^L)\geq\chi(\famG_{i+1})/2$ or $\chi(\famG_{i+1}^R)\geq\chi(\famG_{i+1})/2$.
Assume the former (the other case is analogous).
This and $\chi(\famG_{i+1})>\gamma_{i+1}=2(\gamma_i+2k\xi+i\xi+1)$ imply $\chi(\famG_{i+1}^L)>\gamma_i+2k\xi+i\xi+1$.
Let $\famH_i$ be a subfamily of $\famG_{i+1}^L$ such that $\famG_{i+1}^L\setminus\famH_i\prec\famH_i$ and $\chi(\famH_i)=i\xi+1$.
It can be found greedily, by processing the curves in $\famG_{i+1}^L$ in the order opposite to $\prec$ (from right to left), adding them to $\famH_i$ until $\chi(\famH_i)=i\xi+1$.
Let $\famS_i=\famF_i^L$.
By the definition of $\famG_{i+1}^L$, every curve in $\famH_i$ intersects some curve in $\famS_i$.
It follows that $(\famH_i,\famS_i)$ is a bracket with $\famS_i\prec\famG_{i+1}^L\setminus\famH_i\prec\famH_i$.
Let $I_i$ be the internal region of the bracket $(\famH_i,\famS_i)$.
Let $\famG_i$ be the curves in $\famG_{i+1}^L\setminus\famH_i$ that are entirely contained in $I_i$.
Since $\omega(\famG_{i+1}^L\setminus\famH_i)\leq\omega(\famF)\leq k$, Lemma~\ref{lem:boundary} provides a set of at most $2k$ curves in $\famH_i\cup\famS_i$ such that every curve in $(\famG_{i+1}^L\setminus\famH_i)\setminus\famG_i$ intersects at least one of them.
Since $\famF$ is a $\xi$-family, it follows that $\chi((\famG_{i+1}^L\setminus\famH_i)\setminus\famG_i)\leq 2k\xi$ and thus $\chi(\famG_i)\geq\chi(\famG_{i+1}^L\setminus\famH_i)-2k\xi\geq\chi(\famG_{i+1}^L)-(i\xi+1)-2k\xi>\gamma_i$.
We have thus found families $\famH_i$, $\famS_i$, and $\famG_i$ with the requested properties.

After completing their construction for $1\leq i\leq k$, we choose any $c_0\in\famG_1$, which exists because $\chi(\famG_1)>0$, we choose any $s_0\in\famF_0$ with $s_0\prec\famF_1$ or $\famF_1\prec s_0$ that intersects $c_0$, which exists because $\famG_1\subseteq\famF_1$ and $\famF_1$ is supported from outside in $\famF_0$, and we let $\famH_0=\{c_0\}$ and $\famS_0=\{s_0\}$.
This yields a bracket $(\famH_0,\famS_0)$ with the requested properties and completes the proof of the claim.
\end{proof}

\subsection*{Tree-configurations}

A \emph{binary tree} is a rooted tree in which every node has at most one \emph{left child} and at most one \emph{right child}.
A \emph{descendant} of a node $v$ in such a tree is any node $u$ such that $u\neq v$ and $v$ lies on the path from the root of the tree to $u$.
A \emph{left descendant} of a node $v$ is the left child of $v$ or any descendant of the left child of $v$, and a \emph{right descendant} of $v$ is the right child of $v$ or any descendant of the right child of $v$.
A \emph{tree-configuration} is a sequence $\bigl((x_1,y_1),\ldots,(x_n,y_n)\bigr)$ of pairs of grounded curves such that the following conditions are satisfied:
\begin{itemize}
\item $x_1\prec\cdots\prec x_n\prec y_n\prec\cdots\prec y_1$,
\item for $1\leq i\leq n$, the curves $x_i$ and $y_i$ intersect,
\end{itemize}
and the pairs $(x_1,y_1),\ldots,(x_n,y_n)$ can be arranged into a binary tree $T$ so that
\begin{itemize}
\item $(x_1,y_1)$ is the root of $T$,
\item for $2\leq i\leq n$, the parent of $(x_i,y_i)$ in $T$ is one of $(x_1,y_1),\ldots,(x_{i-1},y_{i-1})$,
\item for $1\leq i<j\leq n$, if $(x_j,y_j)$ is a left descendant of $(x_i,y_i)$, then both $x_j$ and $y_j$ intersect $x_i$, and if $(x_j,y_j)$ is a right descendant of $(x_i,y_i)$, then both $x_j$ and $y_j$ intersect $y_i$.
\end{itemize}
The number $n$ is the \emph{length} of the tree-configuration $\bigl((x_1,y_1),\ldots,(x_n,y_n)\bigr)$, and any tree $T$ satisfying the conditions above is a \emph{witness} of the tree-configuration.
See Figure~\ref{fig:tree-configuration} for an illustration.
The following lemma explains the role of tree-configurations.

\begin{figure}[t]
\centering
\begin{tikzpicture}[yscale=.9]
  \draw plot[smooth,tension=.7] coordinates {(-3.5,0) (-3.4,2.8) (-0.8,4.5) (2.5,4.5) (3.8,3.3) (2.5,1.8)};
  \draw plot[smooth,tension=.7] coordinates {(3.8,0) (3.5,1.8) (2,3.3)};
  \draw plot[smooth,tension=.7] coordinates {(-2.1,0) (-2.5,1.6) (-4.2,3.5) (-2.8,5)};
  \draw plot[smooth,tension=.7] coordinates {(2.6,0) (1.7,1.6) (-2.8,3.5) (-4.1,5.2)};
  \draw plot[smooth,tension=.7] coordinates {(-1.3,0) (-0.3,0.9) (3.4,0.8) (4.4,1.1)};
  \draw plot[smooth,tension=.7] coordinates {(1.5,0) (1.5,1.8) (3,3.3)};
  \draw plot[smooth,tension=.7] coordinates {(-0.4,0) (-0.5,2.3) (0.5,4.2) (0.5,5.4)};
  \draw plot[smooth,tension=.7] coordinates {(0.8,0) (0.6,2.5) (-0.6,4) (-1,5.2)};
  \draw (-4.8,0)--(4.9,0);
  \node[below] at (-3.5,0) {$x_1$};
  \node[below] at (3.8,0) {$y_1$};
  \node[below] at (-2.1,0) {$x_2$};
  \node[below] at (2.6,0) {$y_2$};
  \node[below] at (-1.3,0) {$x_3$};
  \node[below] at (1.5,0) {$y_3$};
  \node[below] at (-0.4,0) {$x_4$};
  \node[below] at (0.8,0) {$y_4$};
\end{tikzpicture}\hskip .8cm
\begin{tikzpicture}[scale=1.3,baseline=-2.4cm]
  \tikzstyle{every node}=[circle,draw,fill,minimum size=3pt,inner sep=0pt]
  \tikzstyle{every label}=[rectangle,draw=none,fill=none,inner sep=3pt]
  \node[label=above:{$(x_1,y_1)$}] (a) at (60:1cm) {};
  \node[label=left:{$(x_2,y_2)$}] (b) at (0,0) {};
  \node[label=right:{$(x_3,y_3)$}] (c) at (1,0) {};
  \node[label=below:{$(x_4,y_4)$}] (d) at (-60:1cm) {};
  \path (a) edge (b) edge (c) (b) edge (d);
\end{tikzpicture}
\caption{A tree-configuration $\bigl((x_1,y_1),(x_2,y_2),(x_3,y_3),(x_4,y_4)\bigr)$ and its unique witness}
\label{fig:tree-configuration}
\end{figure}
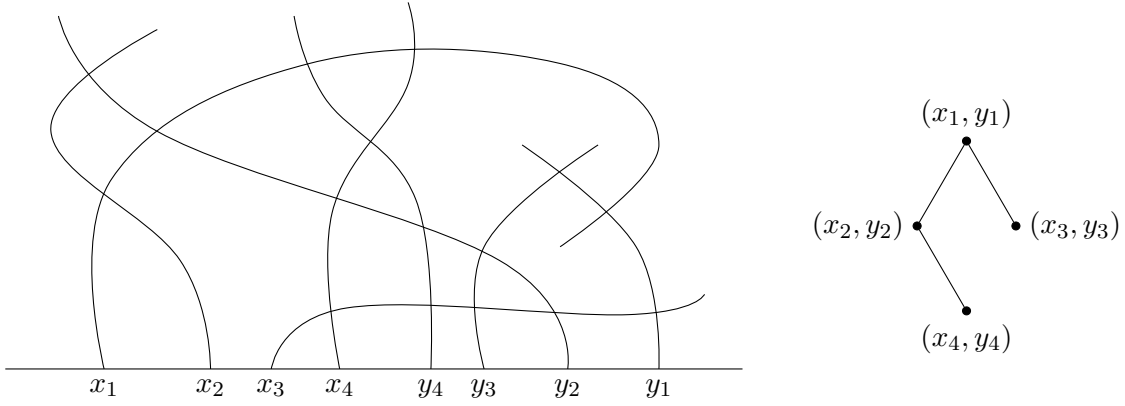

\begin{lemma}
\label{lem:tree-configuration-clique}
Every tree-configuration of length\/ $2^{k-1}$ contains a clique of size\/ $k+1$.
\end{lemma}

\begin{proof}
Let $\bigl((x_1,y_1),\ldots,(x_n,y_n)\bigr)$ be a tree-configuration with $n=2^{k-1}$, and let $T$ be its witness.
Since every binary tree with height less than $k$ has fewer than $2^{k-1}$ nodes, the height of $T$ is at least $k$.
In particular, there are indices $i_1,\ldots,i_k$ with $1=i_1\leq\cdots\leq i_k\leq n$ such that $(x_{i_{r+1}},y_{i_{r+1}})$ is a child of $(x_{i_r},y_{i_r})$ for $1\leq r<k$.
A clique of size $k+1$ arises by taking $x_{i_k}$, $y_{i_k}$, and either $x_{i_r}$ or $y_{i_r}$ for $1\leq r<k$ according to the following rule:
\begin{itemize}
\item if $(x_{i_{r+1}},y_{i_{r+1}})$ is the left child of $(x_{i_r},y_{i_r})$, then take $x_{i_r}$,
\item if $(x_{i_{r+1}},y_{i_{r+1}})$ is the right child of $(x_{i_r},y_{i_r})$, then take $y_{i_r}$.
\end{itemize}
The definition of a tree-configuration guarantees that the respective $x_{i_r}$ or $y_{i_r}$ (whichever is taken) intersects all $x_{i_{r+1}},y_{i_{r+1}},\ldots,x_{i_k},y_{i_k}$.
\end{proof}

The proof of Lemma~\ref{lem:main} proceeds by inductive construction of a tree-configuration of length $2^{k-1}$ in a $\xi$-family of sufficiently large chromatic number, which then contains a clique of size $k+1$, by Lemma~\ref{lem:tree-configuration-clique}.
First, we present a way of extending a tree-configuration by one pair of curves.

An \emph{attachment point} of a binary tree $T$ is a pair $(u,d)$ such that $u$ is a node of $T$ and $d\in\{$left$,{}$right$\}$ is a direction in which $u$ has no child.
If $(u,d)$ is an attachment point of $T$, then a new node can be added to $T$ becoming a child of $u$ in direction $d$.
A \emph{left attachment point} for a node $v$ of $T$ is an attachment point $(u,d)$ such that $(u,d)=(v,{}$left$)$ or $u$ is a left descendant of $v$, and a \emph{right attachment point} for $v$ is an attachment point $(u,d)$ such that $(u,d)=(v,{}$right$)$ or $u$ is a right descendant of $v$.
Thus, a new node added to $T$ at a left or right attachment point for $v$ becomes a left or right descendant of $v$, respectively.
A binary tree with $n$ nodes has exactly $n+1$ attachment points.

Let $T$ be a witness of a tree-configuration $\bigl((x_1,y_1),\ldots,(x_n,y_n)\bigr)$, and let $a$ be an attachment point of $T$.
A grounded curve $c$ is \emph{valid} for $a$ if $x_n\prec c\prec y_n$, $c$ intersects $x_i$ for every node $(x_i,y_i)$ of $T$ whose left attachment point is $a$, and $c$ intersects $y_i$ for every node $(x_i,y_i)$ of $T$ whose right attachment point is $a$.
Any pair of intersecting grounded curves that are valid for the same attachment point of $T$ can be added to the end of $\bigl((x_1,y_1),\ldots,(x_n,y_n)\bigr)$ to form a tree-configuration of length $n+1$ whose witness is obtained from $T$ by adding that pair as a new node at that attachment point.

\begin{lemma}
\label{lem:attachment}
If\/ $T$ is a witness of a tree-configuration\/ $\bigl((x_1,y_1),\ldots,(x_n,y_n)\bigr)$ and\/ $\gamma$ is a cap-curve such that\/ $x_1,y_1,\ldots,x_n,y_n\subset\int\gamma$, then the following statements hold.
\begin{enumerate}
\item\label{item:valid} Every grounded curve\/ $c$ with\/ $x_n\prec c\prec y_n$ that intersects\/ $\gamma$ is valid for at least one attachment point of\/ $T$.
\item\label{item:between} For any pairwise disjoint grounded curves\/ $c_1$, $c_2$, and\/ $c$ with\/ $x_n\prec c_1\prec c\prec c_2\prec y_n$ that intersect\/ $\gamma$, if\/ $c_1$ and\/ $c_2$ are valid for an attachment point\/ $a$ of\/ $T$, then\/ $c$ is also valid for\/ $a$.
\end{enumerate}
\end{lemma}

\begin{proof}
For $1\leq i\leq n$, since $x_i$ and $y_i$ intersect, there is a cap-curve $\nu_i\subseteq x_i\cup y_i$ connecting the basepoints of $x_i$ and $y_i$; it follows that $\gamma\subset\ext\nu_i$ while the basepoint of $c$ lies in $\int\nu_i$, so $c$ intersects $x_i$ or $y_i$.
We find an attachment point $(u,d)$ of $T$ such that $c$ is valid for $(u,d)$ as follows.
We start from the root $(x_1,y_1)$ and repeatedly move to the left child of the current node $(x_i,y_i)$ if $c$ intersects $x_i$ or to the right child if $c$ intersects $y_i$ (choosing any child if $c$ intersects both $x_i$ and $y_i$) until the current node has no child in the requested direction.
In the latter case, we let $u$ be the current node and $d$ be the requested direction.

Now, suppose $c_1$ and $c_2$ intersect a curve $z$ contained in $\int\gamma$.
It follows that there is a cap-curve $\nu\subseteq c_1\cup z\cup c_2$ connecting the basepoints of $c_1$ and $c_2$.
If $c_1\prec c\prec c_2$, then the basepoint of $c$ lies in $\int\nu$, while $\gamma\subset\ext\nu$ (as $c_1,z,c_2\subset\int\gamma$), so $c$ intersects $z$.
This observation applied to every $z\in\{x_1,y_1,\ldots,x_n,y_n\}$ yields the second statement.
\end{proof}

\begin{lemma}
\label{lem:tree-configuration}
There is a function\/ $h\colon\setN\to\setN$ with\/ $h(n)=\smash[t]{2^{O(kn^2)}}\xi^{n-1}$ such that for every\/ $n\in\setN$, every\/ $\xi$-family\/ $\famF$ with\/ $\omega(\famF)\leq k$ and\/ $\chi(\famF)>h(n)$ contains a tree-configuration of length\/ $n$.
\end{lemma}

\begin{proof}
We define the function $h$ by induction, as follows.
We set $h(1)=1$; if $\chi(\famF)>1$, then $\famF$ contains two intersecting curves, which form a tree-configuration of length $1$.
For the induction step, fix $n\geq 1$, and assume that $h(n)$ is defined so that every $\xi$-family $\famH$ with $\chi(\famH)>h(n)$ contains a tree-configuration of length $n$.
Let $g$ be the function claimed by Lemma~\ref{lem:skeleton}.
Let
\begin{equation*}
\alpha=h(n),\qquad\beta=2(n+1)\tbinom{n+2}{2}\xi+2(n+2)\xi,\qquad h(n+1)=\smash[t]{g^{(n+2)}}(2\alpha(\beta+1)),
\end{equation*}
where $\smash[t]{g^{(m)}}$ denotes the $m$-fold composition of $g$.
Let $\famF$ be a $\xi$-family with $\omega(\famF)\leq k$ and $\chi(\famF)>h(n+1)$.
We claim that $\famF$ contains a tree-configuration of length $n+1$.

Let $\famF_0=\famF$.
Lemma~\ref{lem:skeleton} applied $n+2$ times provides families $\famF_1,\ldots,\famF_{n+2}$ and skeletons $(\gamma_1,\famU_1),\ldots,(\gamma_{n+2},\famU_{n+2})$ with the following properties:
\begin{itemize}
\item $\famF=\famF_0\supseteq\famF_1\supseteq\cdots\supseteq\famF_{n+2}$,
\item for $1\leq i\leq n+2$, $(\gamma_i,\famU_i)$ is an $\famF_{i-1}$-skeleton supporting $\famF_i$,
\item for $1\leq i\leq n+2$, we have $\chi(\famF_i)>\smash[t]{g^{(n+2-i)}}(2\alpha(\beta+1))$.
\end{itemize}
In particular, by the last condition, we have $\chi(\famF_{n+2})>2\alpha(\beta+1)$.
Therefore, by Lemma~\ref{lem:mcguinness}, there is a subfamily $\famH\subseteq\famF_{n+2}$ such that $\chi(\famH)>\alpha$ and $\chi(\famF_{n+2}(x,y))>\beta$ for any two intersecting curves $x,y\in\famH$.
Since $\chi(\famH)>\alpha$, there is a tree-configuration $\bigl((x_1,y_1),\ldots,(x_n,y_n)\bigr)$ of length $n$ in $\famH$.
Let $x=x_n$ and $y=y_n$.
Thus $\chi(\famF_{n+2}(x,y))>\beta=2(n+1)\binom{n+2}{2}\xi+2(n+2)\xi$.

Let $\famG$ be the family of curves in $\famF_{n+2}(x,y)$ that intersect some curve in $\famU_i(x,y)$ for every $i$ with $1\leq i\leq n+2$.
If a curve $c\in\famF_{n+2}(x,y)$ intersects no curve in $\famU_i(x,y)$, where $1\leq i\leq n+2$, then $c$ intersects the curve in $\famU_i$ with rightmost basepoint to the left of the basepoint of $x$ (if such a curve exists) or the curve in $\famU_i$ with leftmost basepoint to the right of the basepoint of $y$ (if such a curve exists).
This gives at most $2(n+2)$ curves such that every curve in $\famF_{n+2}(x,y)\setminus\famG$ intersects at least one of them.
The fact that $\famF$ is a $\xi$-family implies $\chi(\famF_{n+2}(x,y)\setminus\famG)\leq 2(n+2)\xi$ and thus $\chi(\famG)\geq\chi(\famF_{n+2}(x,y))-2(n+2)\xi>2(n+1)\binom{n+2}{2}\xi$.

Let $T$ be a witness of the tree-configuration $\bigl((x_1,y_1),\ldots,(x_n,y_n)\bigr)$.
For $1\leq i\leq n+2$, let $\famU_i^a$ be the curves in $\famU_i(x,y)$ that are valid for an attachment point $a$ of $T$.
Thus $\famU_i(x,y)=\bigcup_a\famU_i^a$, by Lemma \ref{lem:attachment}~\ref{item:valid}.
For a triple $\sigma=(i,j,a)$ such that $1\leq i<j\leq n+2$ and $a$ is an attachment point of $T$, let $\famG_\sigma$ be the curves in $\famG$ that intersect a curve in $\famU_i^a$ and a curve in $\smash[b]{\famU_j^a}$.
By the pigeonhole principle, since $T$ has $n+1$ attachment points, we have $\famG=\bigcup_\sigma\famG_\sigma$, and since there are $(n+1)\binom{n+2}{2}$ distinct triples $\sigma$ and $\chi(\famG)>2(n+1)\binom{n+2}{2}\xi$, there is a triple $\sigma=(i,j,a)$ such that $\chi(\famG_\sigma)>2\xi$.

\begin{figure}[t]
\begin{tikzpicture}[scale=.9,yscale=.8]
  \useasboundingbox (0,-0.7) rectangle (15,6.6);
  \coordinate (a0) at (0.6,0);
  \coordinate (a1) at (1.6,5);
  \coordinate (a2) at (5.8,6.2);
  \coordinate (a3) at (13,5.6);
  \coordinate (a4) at (14.4,0);
  \coordinate (b0) at (1.6,0);
  \coordinate (b1) at (2.5,4.3);
  \coordinate (b2) at (6.9,4.2);
  \coordinate (b3) at (12.1,4.7);
  \coordinate (b4) at (13.4,0);
  \coordinate (u0) at (4.8,0);
  \coordinate (u1) at (4.6,3);
  \coordinate (u2) at (3.9,5.6);
  \coordinate (u3) at (5.1,7);
  \coordinate (w0) at (10.3,0);
  \coordinate (w1) at (10.2,2.2);
  \coordinate (w2) at (9.6,4);
  \coordinate (w3) at (10,5.9);
  \coordinate (w4) at (8.8,7);
  \fill[red!5] plot[smooth,tension=.7] coordinates {(a0) (a1) (a2) (a3) (a4)};
  \fill[blue!10] plot[smooth,tension=.7] coordinates {(b0) (b1) (b2) (b3) (b4)};
  \begin{scope}
  \clip plot[smooth,tension=.7] coordinates {(a0) (a1) (a2) (a3) (a4)}--cycle;
  \fill[opacity=.1] plot[smooth,tension=.7] coordinates {(u0) (u1) (u2) (u3) (w4) (w3) (w2) (w1) (w0)}--cycle;
  \end{scope}
  \draw[name path=a,dotted,red] plot[smooth,tension=.7] coordinates {(a0) (a1) (a2) (a3) (a4)};
  \draw[name path=b,dotted,blue] plot[smooth,tension=.7] coordinates {(b0) (b1) (b2) (b3) (b4)};
  \path[name path=uw] plot[smooth,tension=.7] coordinates {(u0) (u1) (u2) (u3) (w4) (w3) (w2) (w1) (w0)};
  \draw[red,intersection segments={of=uw and a,sequence=L1}];
  \draw[red,intersection segments={of=uw and a,sequence=L3}];
  \path[decorate,decoration={markings,mark=at position .5 with {\coordinate (v4);}}] plot[smooth,tension=.7] coordinates {(a0) (a1) (a2) (a3) (a4)};
  \coordinate (v0) at (8.6,0);
  \draw[red] plot[smooth,tension=.7] coordinates {(v0) (8.5,1.7) (8,3.4) (7.6,4.8) (v4)};
  \path[name path=s] plot[smooth,tension=.7] coordinates {(6.1,0) (6.3,2.1) (5.6,4) (6.7,5.2) (8.6,5.4)};
  \draw[dashed,intersection segments={of=s and b,sequence=L2}];
  \draw[blue,intersection segments={of=s and b,sequence=L1}];
  \draw plot[smooth,tension=.7] coordinates {(2.7,0) (2.9,1.9) (3.5,3.8)};
  \coordinate (x0) at (3.6,0);
  \draw plot[smooth,tension=.7] coordinates {(x0) (4.6,1.7) (8.7,2.5) (11.7,3.6)};
  \coordinate (c0) at (7.3,0);
  \draw plot[smooth,tension=.7] coordinates {(c0) (7.1,1.2) (5.7,1.2)};
  \coordinate (y0) at (11.6,0);
  \draw plot[smooth,tension=.7] coordinates {(y0) (11.4,2) (10.7,4.1)};
  \draw plot[smooth,tension=.7] coordinates {(12.5,0) (12.1,1.5) (10.3,1.8) (8.1,3.2) (4.2,2.8) (2.7,3.5)};
  \draw (0,0)--(15,0);
  \node at (4.9,5.4) {$K$};
  \node[below] at (x0) {$x$};
  \node[below] at (u0) {$u_L$};
  \node[below] at (c0) {$c$};
  \node[below] at (v0) {$u$};
  \node[below] at (w0) {$u_R$};
  \node[below] at (y0) {$y$};
  \path[decorate,decoration={pre length=5pt,markings,mark=at position .1 with {\node[left] {$u'$};}},intersection segments={of=s and b,sequence=L1}];
  \path[decorate,decoration={markings,mark=at position .2 with {\node[above left,yshift=-3pt] {$s'$};}},intersection segments={of=s and b,sequence=L2}];
  \path[decorate,decoration={markings,mark=at position .8 with {\node[above right] {$\gamma_i$};}}] plot[smooth,tension=.7] coordinates {(a0) (a1) (a2) (a3) (a4)};
  \path[decorate,decoration={markings,mark=at position .81 with {\node[above right,xshift=-1pt] {$\gamma_j$};}}] plot[smooth,tension=.7] coordinates {(b0) (b1) (b2) (b3) (b4)};
\end{tikzpicture}
\caption{Illustration for the final part of the proof of Lemma~\ref{lem:tree-configuration}}
\label{fig:final-proof}
\end{figure}

The rest of the argument is illustrated in Figure~\ref{fig:final-proof}.
Let $u_L$ and $u_R$ be the curves in $\famU_i^a$ with leftmost and rightmost basepoints, respectively.
Every curve in $\famU_i^a$ lies in the closed region $K$ bounded by $u_L$, $u_R$, the segment of the baseline between the basepoints of $u_L$ and $u_R$, and the part of $\gamma_i$ between its intersection points with $u_L$ and $u_R$.
Since $\famF$ is a $\xi$-family, the curves in $\famG_\sigma$ that intersect $u_L$ have chromatic number at most $\xi$, and so do the curves in $\famG_\sigma$ that intersect $u_R$.
Since $\chi(\famG_\sigma)>2\xi$, there is a curve $c\in\famG_\sigma$ that intersects neither $u_L$ nor $u_R$.
This and the fact that $c$ intersects some curve in $\famU_i^a$ implies $c\subset K$.
The curve $c$ also intersects some curve $u'\in\smash[b]{\famU_j^a}$, which is a subcurve of some curve $s'\in\famF_{j-1}$ valid for $a$.
The facts that $c\subset K$ and $s'\subset\int\gamma_i$ imply $s'\subset K$ unless $s'$ intersects $u_L$ or $u_R$.
Since $s'\in\famF_{j-1}\subseteq\famF_i$ and $\famF_i$ is supported by $(\gamma_i,\famU_i)$, the curve $s'$ intersects some curve $u\in\famU_i$, and by the above, $u$ can be chosen so that $u_L\preceq u\preceq u_R$.
Lemma \ref{lem:attachment}~\ref{item:between} implies that $u$ is also valid for $a$.
The curve $u$ is a subcurve of some curve $s\in\famF_{i-1}$ valid for $a$.
Since $s$ and $s'$ intersect and are both valid for $a$, they can be used to extend the tree-configuration $\bigl((x_1,y_1),\ldots,(x_n,y_n)\bigr)$.
Specifically, if $(x_{n+1},y_{n+1})=(s,s')$ or $(x_{n+1},y_{n+1})=(s',s)$ so that $x_{n+1}\prec y_{n+1}$, then $\bigl((x_1,y_1),\ldots,(x_{n+1},y_{n+1})\bigr)$ is a tree-configuration of length $n+1$ in $\famF$.

It remains to derive the claimed bound on $h$.
By Lemma~\ref{lem:skeleton}, we have $g(\alpha)=\smash[t]{2^{O(k)}}(\xi+\alpha)$, which yields $\smash[t]{g^{(n+2)}}(\alpha)=\smash[t]{2^{O(kn)}}(\xi+\alpha)$.
This yields the following, for $n\geq 1$:
\begin{equation*}
h(n+1)=\smash[t]{g^{(n+2)}}\bigl(\operatorname{poly}(n)\xi h(n)\bigr)=\smash[t]{2^{O(kn)}}\xi+\smash[t]{2^{O(kn)}}\xi h(n)=\smash[t]{2^{O(kn)}}\xi h(n),
\end{equation*}
where the last bound follows from $h(n)\geq 1$.
This and $h(1)=1$ yield $h(n)=\smash[t]{2^{O(kn^2)}}\xi^{n-1}$.
\end{proof}

\begin{proof}[Proof of Lemma~\ref{lem:main}]
Let $\zeta=h(2^{k-1})$, where $h$ is the function claimed by Lemma~\ref{lem:tree-configuration}.
It follows that every $\xi$-family $\famF$ with $\chi(\famF)>\zeta$ contains a tree-configuration of length $2^{k-1}$ and thus, by Lemma~\ref{lem:tree-configuration-clique}, a clique of size $k+1$, which is not possible when $\omega(\famF)\leq k$.
Therefore, every $\xi$-family $\famF$ with $\omega(\famF)\leq k$ satisfies $\chi(\famF)\leq\zeta$.
The bound on $h$ from Lemma~\ref{lem:tree-configuration} immediately yields the claimed bound on $\zeta$.
\end{proof}


\begin{thebibliography}{99}

\bibitem{Ack09}
\href{https://doi.org/10.1007/s00454-009-9143-9}{Eyal Ackerman, On the maximum number of edges in topological graphs with no four pairwise crossing edges, \emph{Discrete Comput. Geom.} 41~(3), 365--375, 2009}.

\bibitem{AAP+97}
\href{https://doi.org/10.1007/BF01196127}{Pankaj~K. Agarwal, Boris Aronov, János Pach, Richard Pollack, and Micha Sharir, Quasi-planar graphs have a linear number of edges, \emph{Combinatorica} 17~(1), 1--9, 1997}.

\bibitem{AG60}
\href{https://doi.org/10.7146/math.scand.a-10607}{Edgar Asplund and Branko Grünbaum, On a coloring problem, \emph{Math. Scand.} 8, 181--188, 1960}.

\bibitem{Ben59}
\href{https://doi.org/10.1073/pnas.45.11.1607}{Seymour Benzer, On the topology of the genetic fine structure, \emph{Proc. Natl. Acad. Sci. USA} 45~(11), 1607--1620, 1959}.

\bibitem{BBD18}
\href{https://doi.org/10.4230/LIPIcs.SWAT.2018.10}{Therese Biedl, Ahmad Biniaz, and Martin Derka, On the size of outer-string representations, in: David Eppstein (ed.), \emph{16th Scandinavian Symposium and Workshops on Algorithm Theory (SWAT 2018)}, vol.~101 of \emph{Leibniz Int. Proc. Inform. (LIPIcs)}, no.~10, pp.~1--14, Schloss Dagstuhl Leibniz-Zentr. Inform., Wadern, 2018}.

\bibitem{BMP-book}
\href{https://doi.org/10.1007/0-387-29929-7}{Peter Brass, William Moser, and János Pach, \emph{Research Problems in Discrete Geometry}, Springer, New York, 2005}.

\bibitem{Bri93}
\href{https://doi.org/10.1007/BF01108825}{Graham~R. Brightwell, On the complexity of diagram testing, \emph{Order} 10~(4), 297--303, 1993}.

\bibitem{Bur65}
James~P. Burling, \emph{On coloring problems of families of polytopes}, PhD thesis, University of Colorado, Boulder, 1965.

\bibitem{CJ17}
\href{http://www.combinatorics.org/ojs/index.php/eljc/article/view/v24i1p33}{Sergio Cabello and Miha Jejčič, Refining the hierarchies of classes of geometric intersection graphs, \emph{Electron. J. Comb.} 24~(1), P1.33, 1--19, 2017}.

\bibitem{CFM+17}
\href{https://doi.org/10.1007/978-3-319-68705-6_12}{Jean Cardinal, Stefan Felsner, Tillmann Miltzow, Casey Tompkins, and Birgit Vogtenhuber, Intersection graphs of rays and grounded segments, in: Hans~L. Bodlaender and Gerhard~J. Woeginger (eds.), \emph{43rd International Workshop on Graph-Theoretic Concepts in Computer Science (WG 2017)}, vol.~10520 of \emph{Lect. Notes Comput. Sci.}, pp.~153--166, Springer, Cham, 2017}.

\bibitem{Cer07}
\href{https://doi.org/10.1016/j.endm.2007.07.072}{Jakub Černý, Coloring circle graphs, \emph{Electron. Notes Discrete Math.} 29, 457--461, 2007}.

\bibitem{Cha11}
\href{https://doi.org/10.1007/978-3-642-22935-0_11}{Parinya Chalermsook, Coloring and maximum independent set of rectangles, in: Leslie Ann Goldberg, Klaus Jansen, Ramamoorthi Ravi, and José~D.~P. Rolim (eds.), \emph{14th International Workshop on Approximation Algorithms for Combinatorial Optimization Problems (APPROX 2011) and 15th International Workshop on Randomization and Computation (RANDOM 2011)}, vol.~6845 of \emph{Lect. Notes Comput. Sci.}, pp.~123--134, Springer, Heidelberg, 2011}.

\bibitem{Dil50}
\href{https://doi.org/10.2307/1969503}{Robert~P. Dilworth, A decomposition theorem for partially ordered sets, \emph{Ann. Math.} 51~(1), 161--166, 1950}.

\bibitem{EET76}
\href{https://doi.org/10.1016/0095-8956(76)90022-8}{Gideon Ehrlich, Shimon Even, and Robert~E. Tarjan, Intersection graphs of curves in the plane, \emph{J. Comb. Theory Ser.~B} 21~(1), 8--20, 1976}.

\bibitem{FP10}
\href{https://doi.org/10.1017/S0963548309990459}{Jacob Fox and János Pach, A separator theorem for string graphs and its applications, \emph{Comb. Probab. Comput.} 19~(3), 371--390, 2010}.

\bibitem{FP12a}
\href{https://doi.org/10.1016/j.ejc.2011.09.021}{Jacob Fox and János Pach, Coloring $K_k$-free intersection graphs of geometric objects in the plane, \emph{Eur. J. Comb.} 33~(5), 853--866, 2012}.

\bibitem{FP12b}
\href{https://doi.org/10.1016/j.aim.2012.03.011}{Jacob Fox and János Pach, String graphs and incomparability graphs, \emph{Adv. Math.} 230~(3), 1381--1401, 2012}.

\bibitem{FP14}
\href{https://doi.org/10.1017/S0963548313000412}{Jacob Fox and János Pach, Applications of a new separator theorem for string graphs, \emph{Comb. Probab. Comput.} 23~(1), 66--74, 2014}.

\bibitem{FPS13}
\href{https://doi.org/10.1137/110858586}{Jacob Fox, János Pach, and Andrew Suk, The number of edges in $k$-quasi-planar graphs, \emph{SIAM J. Discrete Math.} 27~(1), 550--561, 2013}.

\bibitem{Gol77}
\href{https://doi.org/10.1007/BF02253207}{Martin~C. Golumbic, The complexity of comparability graph recognition and coloring, \emph{Computing} 18~(3), 199--208, 1977}.

\bibitem{GRU83}
\href{https://doi.org/10.1016/0012-365X(83)90019-5}{Martin~C. Golumbic, Doron Rotem, and Jorge Urrutia, Comparability graphs and intersection graphs, \emph{Discrete Math.} 43~(1), 37--46, 1983}.

\bibitem{Gya85}
\href{https://doi.org/10.1016/0012-365X(85)90044-5}{András Gyárfás, On the chromatic number of multiple interval graphs and overlap graphs, \emph{Discrete Math.} 55~(2), 161--166, 1985}.\periodsf
\href{https://doi.org/10.1016/0012-365X(86)90224-4}{Corrigendum: \emph{Discrete Math.} 62~(3), 333, 1986}.

\bibitem{Hen98}
\href{http://page.math.tu-berlin.de/~felsner/Diplomarbeiten/hendler.pdf}{Clemens Hendler, Schranken für Färbungs- und Cliquenüberdeckungszahl geometrisch repräsentierbarer Graphen (Bounds for chromatic and clique cover number of geometrically representable graphs), Master's thesis, Freie Universität Berlin, 1998}.

\bibitem{JU17}
\href{https://doi.org/10.1002/jgt.22031}{Svante Janson and Andrew~J. Uzzell, On string graph limits and the structure of a typical string graph, \emph{J. Graph Theory} 84~(4), 386--407, 2017}.

\bibitem{JT-arxiv}
\href{https://arxiv.org/abs/1808.04148}{Vít Jelínek and Martin Töpfer, On grounded L-graphs and their relatives, arXiv:1808.04148}.

\bibitem{KKN04}
\href{http://www.combinatorics.org/ojs/index.php/eljc/article/view/v11i1r52}{Seog-Jin Kim, Alexandr~V. Kostochka, and Kittikorn Nakprasit, On the chromatic number of intersection graphs of convex sets in the plane, \emph{Electron. J. Comb.} 11~(1), R52, 1--12, 2004}.

\bibitem{KR70}
\href{https://doi.org/10.1090/S0002-9939-1970-0253944-9}{Daniel~J. Kleitman and Bruce~L. Rothschild, On the number of finite topologies, \emph{Proc. Am. Math. Soc.} 25~(2), 276--282, 1970}.

\bibitem{Kos88}
\href{http://math.nsc.ru/journals/ti/10/ti_10_0009.pdf}{Alexandr~V. Kostochka, O~verkhnikh otsenkakh khromaticheskogo chisla grafov (On upper bounds for the chromatic number of graphs), in: Vladimir~T. Dementyev (ed.), \emph{Modeli i~metody optimizacii}, vol.~10 of \emph{Trudy Inst. Mat.}, pp.~204--226, Akad. Nauk SSSR SO, Novosibirsk, 1988}.

\bibitem{Kos04}
\href{https://doi.org/10.1090/conm/342/06137}{Alexandr~V. Kostochka, Coloring intersection graphs of geometric figures with a given clique number, in: János Pach (ed.), \emph{Towards a Theory of Geometric Graphs}, vol.~342 of \emph{Contemp. Math.}, pp.~127--138, AMS, Providence, 2004}.

\bibitem{KK97}
\href{https://doi.org/10.1016/S0012-365X(96)00344-5}{Alexandr~V. Kostochka and Jan Kratochvíl, Covering and coloring polygon-circle graphs, \emph{Discrete Math.} 163~(1\nobreakdash--3), 299--305, 1997}.

\bibitem{Kra91a}
\href{https://doi.org/10.1016/0095-8956(91)90090-7}{Jan Kratochvíl, String graphs. I. The number of critical nonstring graphs is infinite, \emph{J. Comb. Theory Ser.~B} 52~(1), 53--66, 1991}.

\bibitem{Kra91b}
\href{https://doi.org/10.1016/0095-8956(91)90091-W}{Jan Kratochvíl, String graphs. II. Recognizing string graphs is NP-hard, \emph{J. Comb. Theory Ser.~B} 52~(1), 67--78, 1991}.

\bibitem{KGK-book}
Jan Kratochvíl, Miroslav Goljan, and Petr Kučera, \emph{String Graphs}, vol. 96~(3) of \emph{Rozpr. ČSAV Řada MPV}, Academia, Prague, 1986.

\bibitem{KM91}
\href{https://doi.org/10.1016/0095-8956(91)90050-T}{Jan Kratochvíl and Jiří Matoušek, String graphs requiring exponential representations, \emph{J. Comb. Theory Ser.~B} 53~(1), 1--4, 1991}.

\bibitem{KPW15}
\href{https://doi.org/10.1007/s00454-014-9640-3}{Tomasz Krawczyk, Arkadiusz Pawlik, and Bartosz Walczak, Coloring triangle-free rectangle overlap graphs with $O(\log\log n)$ colors, \emph{Discrete Comput. Geom.} 53~(1), 199--220, 2015}.

\bibitem{KW17}
\href{https://doi.org/10.1007/s00493-016-3414-x}{Tomasz Krawczyk and Bartosz Walczak, On-line approach to off-line coloring problems on graphs with geometric representations, \emph{Combinatorica} 37~(6), 1139--1179, 2017}.

\bibitem{LMPT94}
\href{https://doi.org/10.1112/blms/26.2.132}{David Larman, Jiří Matoušek, János Pach, and Jenő Törőcsik, A Ramsey-type result for convex sets, \emph{Bull. London Math. Soc.} 26~(2), 132--136, 1994}.

\bibitem{LMPW14}
\href{https://doi.org/10.1007/s00454-014-9614-5}{Michał Lasoń, Piotr Micek, Arkadiusz Pawlik, and Bartosz Walczak, Coloring intersection graphs of arc-connected sets in the plane, \emph{Discrete Comput. Geom.} 52~(2), 399--415, 2014}.

\bibitem{Lee17}
\href{https://doi.org/10.4230/LIPIcs.ITCS.2017.1}{James~R. Lee, Separators in region intersection graphs, in: Christos~H. Papadimitriou (ed.), \emph{8th Innovations in Theoretical Computer Science Conference (ITCS 2017)}, vol.~67 of \emph{Leibniz Int. Proc. Inform. (LIPIcs)}, no.~1, pp.~1--8, Schloss Dagstuhl Leibniz-Zentr. Inform., Wadern, 2017}.

\bibitem{Lov83}
László Lovász, Perfect graphs, in: Lowell~W. Beineke and Robin~J. Wilson (eds.), \emph{Selected Topics in Graph Theory}, vol.~2, pp.~55--87, Academic Press, London, 1983.

\bibitem{MPW98}
\href{https://doi.org/10.1002/(SICI)1097-0037(199808)32:1<13::AID-NET2>3.0.CO;2-M}{Ewa Malesińska, Steffen Piskorz, and Gerhard Weißenfels, On the chromatic number of disk graphs, \emph{Networks} 32~(1), 13--22, 1998}.

\bibitem{Mat14}
\href{https://doi.org/10.1017/S0963548313000400}{Jiří Matoušek, Near-optimal separators in string graphs, \emph{Comb. Probab. Comput.} 23~(1), 135--139, 2014}.

\bibitem{McG96}
\href{https://doi.org/10.1016/0012-365X(95)00316-O}{Sean McGuinness, On bounding the chromatic number of L-graphs, \emph{Discrete Math.} 154~(1\nobreakdash--3), 179--187, 1996}.

\bibitem{McG00}
\href{https://doi.org/10.1007/PL00007228}{Sean McGuinness, Colouring arcwise connected sets in the plane~I, \emph{Graphs Comb.} 16~(4), 429--439, 2000}.

\bibitem{MP93}
\href{https://doi.org/10.1016/0012-365X(93)90176-T}{Matthias Middendorf and Frank Pfeiffer, Weakly transitive orientations, Hasse diagrams and string graphs, \emph{Discrete Math.} 111~(1\nobreakdash--3), 393--400, 1993}.

\bibitem{MWW-arxiv}
\href{https://arxiv.org/abs/1802.09969}{Torsten Mütze, Bartosz Walczak, and Veit Wiechert, Realization of shift graphs as disjointness graphs of $1$-intersecting curves in the plane, arXiv:1802.09969}.

\bibitem{NR95}
\href{http://hdl.handle.net/10338.dmlcz/118756}{Jaroslav Nešetřil and Vojtěch Rödl, More on the complexity of cover graphs, \emph{Comment. Math. Univ. Carolin.} 36~(2) 269--278, 1995}.

\bibitem{PRT06}
\href{https://doi.org/10.1007/978-3-540-32439-3_12}{János Pach, Radoš Radoičić, and Géza Tóth, Relaxing planarity for topological graphs, in: Ervin Győri, Gyula~O.~H. Katona, and László Lovász (eds.), \emph{More Graphs, Sets and Numbers}, vol.~15 of \emph{Bolyai Soc. Math. Stud.}, pp.~285--300, Springer, Heidelberg, 2006}.

\bibitem{PSS96}
\href{https://doi.org/10.1007/BF02086610}{János Pach, Farhad Shahrokhi, and Mario Szegedy, Applications of the crossing number, \emph{Algorithmica} 16~(1), 111--117, 1996}.

\bibitem{PT-arxiv}
\href{https://arxiv.org/abs/1811.09158}{János Pach and István Tomon, On the chromatic number of disjointness graphs of curves, arXiv:1811.09158}.

\bibitem{PT02}
\href{https://doi.org/10.1007/s00454-002-2891-4}{János Pach and Géza Tóth, Recognizing string graphs is decidable, \emph{Discrete Comput. Geom.} 28~(4), 593--606, 2002}.

\bibitem{PT06}
\href{https://doi.org/10.1007/s00493-006-0032-z}{János Pach and Géza Tóth, How many ways can one draw a graph?, \emph{Combinatorica} 26~(5) 559--576, 2006}.

\bibitem{PKK+13}
\href{https://doi.org/10.1007/s00454-013-9534-9}{Arkadiusz Pawlik, Jakub Kozik, Tomasz Krawczyk, Michał Lasoń, Piotr Micek, William~T. Trotter, and Bartosz Walczak, Triangle-free geometric intersection graphs with large chromatic number, \emph{Discrete Comput. Geom.} 50~(3), 714--726, 2013}.

\bibitem{PKK+14}
\href{https://doi.org/10.1016/j.jctb.2013.11.001}{Arkadiusz Pawlik, Jakub Kozik, Tomasz Krawczyk, Michał Lasoń, Piotr Micek, William~T. Trotter, and Bartosz Walczak, Triangle-free intersection graphs of line segments with large chromatic number, \emph{J. Comb. Theory Ser.~B} 105, 6--10, 2014}.

\bibitem{Pee91}
\href{https://pure.uvt.nl/portal/en/publications/on-coloring-junit-sphere-graphs(0678289b-1798-4adb-9ca1-9ef5d608d166).html}{René Peeters, On coloring $j$-unit sphere graphs, Technical Report FEW 512, Department of Economics, Tilburg University, 1991}.

\bibitem{RW-inpress}
\href{https://doi.org/10.1007/s00454-018-0031-z}{Alexandre Rok and Bartosz Walczak, Coloring curves that cross a fixed curve, \emph{Discrete Comput. Geom.}, in press, doi:10.1007/s00454-018-0031-z}.

\bibitem{SSS03}
\href{https://doi.org/10.1016/S0022-0000(03)00045-X}{Marcus Schaefer, Eric Sedgwick, and Daniel Štefankovič, Recognizing string graphs in NP, \emph{J. Comput. Syst. Sci.} 67~(2), 365--380, 2003}.

\bibitem{SS04}
\href{https://doi.org/10.1016/j.jcss.2003.07.002}{Marcus Schaefer and Daniel Štefankovič, Decidability of string graphs, \emph{J. Comput. Syst. Sci.} 68~(2), 319--334, 2004}.

\bibitem{Sin66}
\href{https://doi.org/10.1002/j.1538-7305.1966.tb01713.x}{Frank~W. Sinden, Topology of thin film RC circuits, \emph{Bell Syst. Tech. J.} 45~(9), 1639--1662, 1966}.

\bibitem{Suk14}
\href{https://doi.org/10.1007/s00493-014-2942-5}{Andrew Suk, Coloring intersection graphs of $x$-monotone curves in the plane, \emph{Combinatorica} 34~(4), 487--505, 2014}.

\bibitem{SW15}
\href{https://doi.org/10.1016/j.comgeo.2015.06.001}{Andrew Suk and Bartosz Walczak, New bounds on the maximum number of edges in $k$-quasi-planar graphs, \emph{Comput. Geom.} 50, 24--33, 2015}.

\bibitem{Val97}
\href{https://doi.org/10.1007/3-540-63938-1_63}{Pavel Valtr, Graph drawing with no $k$ pairwise crossing edges, in: Giuseppe Di Battista (ed.), \emph{5th International Symposium on Graph Drawing (GD 1997)}, vol.~1353 of \emph{Lect. Notes Comput. Sci.}, pp.~205--218, Springer, Heidelberg, 1997}.

\end{thebibliography}
\end{document}